\documentclass[11pt]{amsart}

\usepackage{amsthm}
\usepackage{amsmath}
\usepackage{graphicx}
\usepackage{amsfonts}
\usepackage{amssymb}
\usepackage{latexsym}
\usepackage{color}
\usepackage{subfigure}
\usepackage{amscd}
\usepackage{hyperref}
\usepackage{booktabs}
\usepackage{enumerate}

\usepackage[T1]{fontenc}
\usepackage{tgtermes}   

\graphicspath{{img/}}

\newtheorem{theorem}{Theorem}[section]

\newtheorem{corollary}[theorem]{Corollary}
\newtheorem{lemma}[theorem]{Lemma}
\newtheorem{proposition}[theorem]{Proposition}
\newtheorem{conjecture}[theorem]{Conjecture}

\theoremstyle{definition}
\newtheorem{definition}[theorem]{Definition}

\newtheorem{remark}[theorem]{Remark}
\newtheorem{example}[theorem]{Example}

\newcommand{\R}{\mbox{$\mathbf R$}}

\newcommand{\C}{\mbox{$\mathcal C$}}

\newcommand{\w}{\mathcal{W}}

\newcommand{\ww}{\mathbf{w}}

\renewcommand{\v}[1]{{#1}}

\newcommand{\ca}{\gamma}
\newcommand{\eca}{\gamma'}
\newcommand{\epc}{\gamma_0}

\begin{document}
\title{A geometric and combinatorial view of weighted voting}
\address{Department of Mathematics, Wake Forest University, Winston-Salem, NC 27109}
\author{Sarah Mason}
\email{masonsk@wfu.edu}

\author{Jason Parsley}
\email{parslerj@wfu.edu}

\date{\today}


\maketitle

\begin{abstract}
 A natural partial ordering exists on the set of all weighted games and, more broadly, on all linear games.  We describe several properties of the partially ordered sets formed by these games  and utilize this perspective to enumerate proper linear games with one generator.  We introduce a geometric approach to weighted voting by considering the convex polytope of all possible realizations of a weighted game and connect this geometric perspective to the weighted games poset in several ways.  In particular, we prove that generic vertical lines in $\C_n$, the union of all weighted $n$-player polytopes, correspond to maximal saturated chains in the poset of weighted games, i.e., the poset is a blueprint for how the polytopes fit together to form $\C_n$.  We show how to compare the relationships between the powers of the players using the polytope directly.  Finally, we describe the facets of each polytope, from which we develop a method for determining the weightedness of any linear game that covers or is covered by a weighted game.
\end{abstract}

\section{Introduction}
\emph{Weighted voting} refers to the situation where $n$ players, each with a certain weight, vote on a yes or no motion.  For one side to win, the weights of its players (frequently referred to as players) must reach a certain fixed quota $q$.  A natural example is a corporation:  each stockholder is a player with weight equal to the shares of stock he or she owns.  The goal of this article is to describe certain combinatorial and geometric structures of weighted voting and to detail the connections between these viewpoints.  

Weighted voting forms an important class of \emph{simple games} (cf.\  Definition~\ref{def:simple}), whose framework offers several different interpretations.  Simple games may be viewed as  hypergraphs~\cite{GilThesis,Isb58} and also as logic gates~\cite{MTB70, Mur71}; in these situations, weighted voting corresponds to threshold graphs and threshold functions, respectively.  We suggest the excellent book~\cite{TZ99} as a first reference on simple games for the nonexpert reader.

When the players in a simple game are totally ordered, the game is called \emph{linear}~\cite{FrePon10} (or \emph{directed}~\cite{KS95} or \emph{complete}~\cite{FrePue08}).  All weighted games are linear (but not vice-versa), since the players' weights provide a natural ordering.  Taylor and Zwicker~\cite{TZ92} have characterized linear games via swap robustness and weighted games via trade robustness; we utilize this trading approach in Section~\ref{sec:poset}.

A total ordering on the players in a linear game naturally leads to a partial ordering on coalitions.   
For $n$ players, the coalitions form the well-known partially ordered set (\emph{poset}) denoted $M(n)$.  Stanley has shown~\cite{Sta80} that this poset is rank-unimodal and exhibits the Sperner property.  
Linear games are in one-to-one correspondence with the filters of $M(n)$, which form the filtration poset $J(M(n))$ or $J_n$.  The generators of the filter corresponding to a game are its shift-minimal winning coalitions.  Krohn and Sudh\"{o}lter~\cite{KS95} introduce this filtration partial ordering on simple linear games and weighted games and investigate several consequences of the Sperner property.  They then use linear programming methods to obtain efficient algorithms which test whether or not a linear game is weighted.  As a subposet of $J_n$, we construct the posets $\w_n$ of all weighted games and $\Pi_n$ of all proper linear games in the style of Krohn and Sudh\"{o}lter.  All of these posets are symmetric, ranked lattices.   We prove in Proposition~\ref{prop:induced} that the weighted games poset $\w_n$ is an induced subposet of $J_n$.
Figure~\ref{fig:poset4} depicts $M(3), M(4)$, and the top half of $J_4$.   

In Section~\ref{sec:geometry}, we introduce a geometric approach to weighted games.  By scaling the weights to sum to one, we define the simplex of normalized weights $\Delta_n$ and the configuration region $\C_n = (0,1] \times \Delta_n$, which depicts all realizations of $n$-player weighted games in quota-weight space.   Each coalition corresponds to a half-space intersecting $\C_n$; each weighted game corresponds to the convex polytope in $\C_n$ formed by intersecting the half-spaces arising from all coalitions winning in that game.   We show in Section~\ref{sec:polytope} that each polytope is indeed convex, $n$-dimensional, closed on the top and side facets, and open on the bottom facets.  

This geometric approach is quite different from the classical `separating hyperplanes' approach, in which coalitions represent vertices of the $n$-dimensional unit cube.  A linear game is weighted if the sets of vertices which correspond to its winning coalitions may be separated by appropriate hyperplanes from the remaining vertices (the losing coalitions) of the cube.  See~\cite{EinLehRegular, Mur71, TZ99} for details.

Players can be broken into equivalence classes of equally desirable players, called \emph{hierarchies}.  Each possible hierarchy of players is associated to a subsimplex of $\Delta_n$.   We show in Theorem~\ref{thm:hierarchy} that the hierarchy for a weighted game corresponds to the smallest subsimplex onto which its polytope projects.  As a corollary, we characterize symmetric games as the only ones that project onto vertices of $\Delta_n$.  

Theorem~\ref{thm:satchain}, the first of two main results of this article, connects the above geometric approach to the poset of weighted games.  We show that for a generic choice of weights, \emph{moving vertically through the configuration region traverses a maximal saturated chain in $\w_n$}.  In other words, we may view the polytopes for $n$ players as building blocks, and the ordering in $\w_n$ provides instructions on how to stack them so as to construct $\C_n$.  Every vertical line in the configuration region therefore describes a saturated chain in the poset $\w_n$.  However, not every saturated chain in the poset $\w_n$ corresponds to a vertical line in $\C_n$.  That is, there exist saturated chains in $\w_n$ that cannot be realized by fixing a collection of weights and varying the quota.  We conjecture a necessary and sufficient condition, \emph{inequality robustness} to determine whether or not a saturated chain can be realized through fixed weights; see Conjecture~\ref{conj:chain}.  

In Theorem~\ref{thm:facets}, we describe the correlation between facets, hierarchies, and posets.  Furthermore, we prove that a weighted game's polytope has $n-k+d$ facets, where $k$ is the number of nontrivial symmetry classes of players in the game and $d$ is the degree of the game as a vertex in $\w_n$.  

The second main result of this article, Theorem~\ref{thm:method}, provides a method for determining the weightedness of a linear game covering or covered by a weighted game in $J_n$.  This method reduces to a linear programming problem which is different and possibly simpler than the standard linear programming approaches (cf.\ \cite{FK11, KS95}) for determining weightedness.

This article is organized as follows.  Section~\ref{sec:background} describes the relevant background on weighted and simple games and assumes little expertise with voting theory.  Section~\ref{sec:poset} contains our combinatorial approach via the partial orderings on coalitions, linear games, and weighted games.  In Section~\ref{sec:geometry}, we detail the geometry of weighted games and its connections to the aforementioned posets.  We conclude by describing relevant open problems and future work in Section~\ref{sec:future}.

\section{Background}  \label{sec:background}

\subsection{Weighted games}

Weighted games, also known as \emph{weighted voting systems}, belong to a much larger class known as \emph{simple games}.  To understand them, we require some preliminary definitions.

Throughout this paper we restrict to a finite set of $n$ players (or voters), who vote yes or no on a motion.  The set of players who vote the same on a given motion is known as a \emph{coalition}.  The set $[n]=\{1,2,\ldots, n\}$ of all players is called the \emph{grand coalition}.  In weighted voting, a coalition is \emph{winning} if the sum of its weights is greater than or equal to the quota.  A \emph{minimal winning coalition} is one possessing no winning coalition as a proper subset; if any player leaves such a coalition the resulting coalition will no longer be winning.  A \emph{dictator} has weight greater than or equal to the quota.  A \emph{dummy} is a player appearing in no minimal winning coalitions.

A first, straightforward observation is that the weights of players can be misleading in understanding weighted voting.  The winning coalitions determine everything about  a weighted game.  For instance, the sets of weights $w_3 = w_2 =0.49, w_1 =  0.02$ and $w_3' = w_2' = w_1' = 1/3$ are vastly different but at a quota of $0.51$ produce the same winning coalitions.  Both of these represent a simple majority game in which any two players can win by voting together.

\begin{definition} \label{def:simple}
A \emph{simple game} $g$ is a pair $\left([n],W_g \right)$ in which 
\begin{itemize}
\item $[n]=\{1,2, \hdots , n\}$ is a finite set of players,
\item $W_g$ is a collection of subsets of $[n]$ that represent the winning coalitions for $g$, and
\item $[n] \in W_g$ while $\emptyset \notin W_g$,
\end{itemize}
that satisfies the \emph{monotonicity property}:  if a coalition is winning, then its supersets must also be winning, i.e., if $S \in W_g$ and $ S \subseteq R \subseteq [n]$, then $R \in W_g$. 
\end{definition}

\begin{example}{\label{ex:maj}}
For example, a simple majority game $g$ with three players, namely $[3]=\{ 1,2,3\}$, in which any two players can win by voting together is given by $W_g=\{32, 31, 21, 321\}$.  Thus, $g=([3], \{32, 31, 21, 321\})$.  \hfill{$\Diamond$}
\end{example}

Henceforth, we assume that all games considered are simple.  Also, we note that many authors choose to omit the requirement that $[n] \in W_g$ while $\emptyset \notin W_g$.  Our  choice is advantageous from a voting and a geometric perspective but not from a combinatorial perspective, as we discuss in Remark~\ref{rmk:allwin}.

\begin{definition} 
A simple game is \emph{weighted} if there exist weights $w_i \in [0,\infty)$ and a quota $q \in (0,\sum w_i]$ such that coalition $A$ is winning if and only if the sum $w_A$ of its weights is greater than or equal to the quota.  The vector $(q: \ww) = (q: w_n, \ldots, w_2, w_1)$ is said to \emph{realize} (or \emph{represent}, we use these terms synonymously) $v$ as a weighted game. 
\end{definition}

The game $g$ appearing in Example~\ref{ex:maj} is weighted since it is realized by the vector $(60: 34, 33, 33)$.  See Appendix~\ref{app:unweighted6} for a list of unweighted games with $6$ players.

Contrary to much of the voting literature, we enumerate players by \emph{increasing} weight, in order to easily determine the rank of a coalition in the poset $M(n)$ (cf.\  Section~\ref{sec:mn}).  We shall refer to players by the corresponding ordinals: player $\v{n}$ has the greatest weight, player $\v{n-1}$ has the next greatest, \ldots , player $\v{1}$ has the lowest weight.   We follow the convention that the weight vector $\ww=(w_n, \ldots, w_1)$ lists the weights in decreasing order.

Note that weighted voting is scale-invariant:  multiplying each of the weights and the quota by a positive constant does not change the winning coalitions.  Therefore, we may normalize the weights so that they sum to 1.  Define a \emph{normalized weight} to be a vector $\ww=(w_n,\ldots, w_1)$  that satisfies
\begin{equation} \label{w-order}
w_n \geq w_{n-1} \geq \ldots \geq w_1 \geq 0, \qquad \sum_{i=1}^n w_i = 1 .
\end{equation}
We denote the set of normalized weights in $\R^n$ as $\Delta_n$, since it forms an $(n-1)$-dimensional simplex with vertices 
$p_1 = (1,0,0,0, \ldots)$, \; $p_2 = \left(\tfrac{1}{2}, \tfrac{1}{2}, 0, 0, \ldots \right)$, \; $p_3  =  \left(\tfrac{1}{3}, \tfrac{1}{3}, \tfrac{1}{3}, 0, 0, \ldots \right)$, \; \ldots,  \; $p_n =  \left(\tfrac{1}{n}, \tfrac{1}{n}, \ldots, \tfrac{1}{n}\right)$.

\begin{remark}
Geometrically, ordering the weights restricts their geometry from the orthant\footnote{The term \emph{orthant} is the $n$-dimensional analogue of a quadrant or octant.} 
$\R_{\geq 0}^n\setminus~\{\mathbf{0}\}$ to the closure (in the subspace topology), minus the origin, of one particular component of the configuration space $\mathtt{C}_n\left( (0, \infty) \right)$.  (The configuration space $\mathtt{C}_n(X)$ consists of all ordered $n$-tuples of distinct $x_i \in X$.)  This closure produces an unbounded polytope of infinite rays, diffeomorphic to $\Delta_n \times (0,\infty)$.    Normalizing the weights deformation retracts this space onto the compact simplex $\Delta_n$.
\hfill{$\Diamond$}
\end{remark}

On $\Delta_n$, we use coordinates $\{w_n, w_{n-1}, \ldots, w_2\}$  and view $w_1$ as a dependent variable equal to $1-w_n  - \ldots - w_2$.

\begin{definition}
For $n$ players, the \emph{configuration region} $\C_n$ is the space of all realizations $(q:\ww)$ of weighted games,  that is,
\begin{equation}
\C_n := (0,1] \times \Delta_n \subset \R^{n+1}.
\end{equation}
\end{definition}

In Section~\ref{sec:geometry}, we study the geometry of $\C_n$  as an approach toward understanding weighted voting.  Note that the set of weighted games on $n$ players may be equivalently defined as the nonempty equivalence classes of points in $\C_n$ where two points are equivalent if they produce the same winning coalitions.

\subsection{Background on simple games}{\label{sec:games}}
We will require a few definitions regarding simple games, which we provide here.  We begin with the desirability relation on players, which was introduced in~\cite{Isb58} and generalized in~\cite{MasPel66}.   (See also~\cite{Mur71}.)

\begin{definition}
Let $([n],W)$ be a simple game.  We say that player $\v{i}$ is \emph{more desirable} than player $\v{j}$ (denoted $\v{i} \succeq \v{j}$) in $([n],W)$ if 
$$S \cup \{ \v{j} \} \in W \Rightarrow S \cup \{ \v{i} \} \in W, \qquad \rm{for} \; \rm{all \; coalitions} \; S \subseteq [n] \setminus \{\v{i},\v{j}\}.$$  
We say that players $\v{i}$ and $\v{j}$ are \emph{equally desirable} (denoted $\v{i} \sim \v{j}$) in $([n], W)$ if 
$$S \cup \{ \v{j} \} \in W \iff S \cup \{ \v{i} \} \in W, \qquad \rm{for} \; \rm{all \; coalitions} \; S \subseteq [n] \setminus \{\v{i},\v{j}\}.$$
Any simple game with a totally ordered desirability relation is called \emph{linear} (or \emph{directed} or \emph{complete}.)
\end{definition}

Each linear game breaks the players into equivalence classes of equally desirable players; this decomposition is called a \emph{hierarchy} and the equivalence classes are sometimes referred to as \emph{symmetry classes}.  The \emph{power composition}, $\gamma_v$, of a linear game $v$ with $k$ nontrivial equivalence classes in its hierarchy is the composition $\gamma_v:=(a_1,a_2, \hdots , a_k)$ where $a_i$ is the number of players in equivalence class $i$.  That is, if a game has $n_0$ dummies, $n_1$ players in its strongest class, $n_2$ in the next strongest, down to $n_k$ in its weakest nontrivial class, then it has power composition $(n_1, n_2, \ldots, n_k)$, which is a composition of $n-n_0$ into $k$ parts.  The power composition is frequently denoted in the literature by $\overline{n}$.  See Carreras and Freixas~\cite{CarFreComplete} or Freixas and Molinero~\cite{FreMol09} for several important uses of the power composition as a vector.  We note that any power index which is monotone, i.e., respects the desirability ordering, must distribute power according to this composition.

For example, suppose for a linear game $v$ on $7$ players that $\v{7} \succ \v{6} \sim \v{5} \succ \v{4} \succ \v{3} \sim \v{2} \sim \v{1}$; then there are four symmetry classes of players.  We may express its hierarchy either as the string $(\succ \sim \succ \succ \sim \sim)$ or using its power composition $(1,2,1,3)$.

\begin{definition}
One may view any linear game on $m$ players as inducing a linear game, called the \emph{induced game}, on $n>m$ players by simply adding $n-m$ dummies.  The $m$th strongest player remains $m$th, which means in our notation we add $n-m$ to each player's number.  Linear games share most properties with their induced games, including power compositions, weightedness and properness.  See Example~\ref{ex:induced}.
\end{definition}

\begin{example}~\label{ex:inducedA}
Consider the weighted game $g$ on 4 players realized by $(0.6:  0.35, 0.25, 0.2, 0.2)$ and specified by its six winning coalitions, $W_v = \{321, 421, 43, 431, 432, 4321\}$.  Its induced game on $5$ players has 12 winning coalitions: $\{432,532,54,542,543,5432,4321,5321,541,5421,5431,54321\}$; player $1$ is a dummy.  \hfill{$\Diamond$}
\end{example}

We will study partial orderings of both linear and weighted games in Section~\ref{sec:poset}.  For now, let us consider the ordering on players in a weighted game.

For a weighted game $v$, the desirability ordering weakly respects the ordering by weights, as follows.  If $w_i=w_j$, then $\v{i} \sim \v{j}$.  If $w_i > w_j$, then $\v{i} \succeq \v{j}$, since the weight of $\{ j \} \cup S$ ($i,j \notin S$) will be strictly less than the weight of $\{ i \} \cup S$, and hence $\{ j \} \cup S \in W_v$ implies $ \{ i \} \cup S \in W_v$.

The desirability ordering for a weighted game is a total ordering on the players, and hence all weighted games are linear.  Not all linear games are weighted though.  The first examples occur for $n=6$ players, where 60 of the 1171 linear games fail to be weighted; for reference, we list these in Appendix~\ref{app:unweighted6}.

\subsection{Duality}  
Linear games feature a useful notion of duality.  Each linear game has a unique dual linear game, with which it shares many important properties:  weightedness, power compositions, and as we will prove in Theorem~\ref{thm:dualgeom}, a geometric congruence.

\begin{definition}
For a simple game $v$, we define its \emph{dual game} $v^*$ by specifying its winning coalitions:
\begin{center}
coalition $A$ is winning in $v^* \quad \Leftrightarrow \quad$ coalition $A^c$ is losing in $v$.
\end{center}
A linear game $v$ is \emph{self-dual} if and only if $v$ equals its dual game $v^*$.  
\end{definition}
 
In many voting contexts, there is a restriction that at most one side may win, i.e., no coalition and its complement are both winning.  One rationale for this restriction is that if opposing sides (complementary coalitions) could both win, a decision would not be reachable and the result of the process would be a stalemate.  

\begin{definition}
A simple game is said to be \emph{proper} if for each complementary pair $\{A, A^c\}$ of coalitions, at most one is winning; otherwise it is \emph{improper}.  A simple game is said to be \emph{strong} if for each complementary pair $\{A, A^c\}$ of coalitions, at least one is winning.  
\end{definition}

Observe that the dual of a proper game is strong, and vice-versa.  Furthermore, a game is self-dual if and only if it is both proper and strong, i.e., for each complementary pair $\{A, A^c\}$ of coalitions, precisely one is winning.  This terminology has evolved as the literature on this subject has grown to include many different viewpoints; we largely follow \cite{TZ99}.  For example, self-dual games have been referred to as \emph{constant-sum} games by von Neumann and Morgenstern~\cite{vNeMorBook}, \emph{strong} games by Isbell~\cite{Isb59}, and \emph{zero-sum} games by Krohn and Sudh\"{o}lter~\cite{KS95}.  The term \emph{self-dual} is most descriptive in our context since the notion of duality plays an important role in our approach.

\begin{example}
Consider the game $v$ on $3$ players given by winning coalitions $W_v=\{21,31,32,321\}$.  The coalitions which are losing in $v$ are $\{\emptyset, 1,2,3\}$.  Their complements, the winning coalitions in $v^*$, are $\{321,32,31,21\}$, precisely the winning coalitions in $v$.  Therefore $v$ is self-dual.  

Next consider the game $u$ on $3$ players given by winning coalitions $W_u=\{31, 32,321\}$.  The coalitions which are losing in $u$ are $\{ \emptyset, 1,2,3,21\}$. Their complements, the winning coalitions in $u^*$, are $\{ 321, 32,31,21,3\}$.  Note that the winning coalitions in $u^*$ are not disjoint from the winning coalitions in $u$, but they are not equal.  Therefore $u$ is not self-dual, but it is proper.
\hfill{$\Diamond$}
\end{example}


\section{Weighted and linear games as partially ordered sets} 
\label{sec:poset}

A \emph{partially ordered set}, or \emph{poset}, $P$, is a set equipped with a binary ordering $\le$ which is reflexive, antisymmetric, and transitive.  The ordering is partial (as opposed to total) since not all elements of the set must be comparable under the ordering.  An element $y \in P$ is said to \emph{cover} another element $x \in P$ if $y > x$ and there is no $z$ such that $y > z > x$.

A poset $P$ is said to be \emph{ranked} (or \emph{graded}) if there exists a \emph{rank function} $\rho: P \rightarrow \mathbb{Z}_{\geq 0}$ compatible with the partial ordering 
such that $\rho(y) = \rho(x) + 1$ if $y$ covers $x$.  The value $\rho(x)$ is called the \emph{rank} of the element $x$.  The numbers of elements of each rank can be organized into a \emph{rank-generating function} given by
$$\sum_{r=0}^{m} a_r q^r,$$ 
where $m$ is the maximal rank and $a_r$ is the number of elements of rank $r$.

In this section, we describe four different posets associated to weighted voting and linear games.  We begin in Section~\ref{sec:mn} with the well-known poset $M(n)$ which represents the ordering on coalitions within an ordered set of $n$ players.  Then we argue that the linear games on $n$ players are in bijection with the filters of $M(n)$.  These filters possess a partial ordering, from which we form a poset of linear games $J_n$.  We also construct subposets representing all weighted games and all proper linear games.

\subsection{An ordering on coalitions}
\label{sec:mn}

We label a coalition of players by listing the indices of the players represented in the coalition as a decreasing sequence.  For example, the coalition formed by players $\v{5}, \v{4}$, and $\v{2}$ is denoted $\{5,4,2\}$ or merely $542$ when clear.  

\begin{definition}{\label{partialorder}}
The \emph{coalitions poset}, denoted $M(n)$, consists of all possible coalitions of $n$ players under the following ordering.  Coalition $B=\{b_1,b_2, \hdots , b_j\}$ is greater than or equal to coalition $A=\{a_1, a_2, \hdots , a_k\}$ if and only if $j \geq k$ and $b_i \geq a_i$ for each $1 \le i \le k$.
We denote this $B \ge A$.
\end{definition}

Definition {\ref{partialorder}} produces a partial ordering on the set of all coalitions formed by $n$ players since it satisfies the reflexivity, antisymmetry, and transitivity conditions.  Krohn and Sudh\"{o}lter~\cite{KS95} use this ordering to count weighted and linear games.  Figure~\ref{fig:poset4} depicts the coalitions posets $M(3)$ and $M(4)$. 

The poset $M(n)$ (A distributive lattice) appears in several other settings and has many interesting combinatorial properties.  In particular, the rank-generating function of $M(n)$ is $\prod_{i=1}^n (1+q^i)$, which was proven to have unimodal coefficients by Hughes~\cite{Hug77}.  The structure of $M(n)$ was shown by Lindstr\"{o}m~\cite{Lin70} to be related to a conjecture of Erd\"{o}s and Moser~\cite{Ent68, Erd65, SarSze65}.  Stanley uses the Coxeter system structure of type $B_n$ to obtain a different construction of the poset $M(n)$.  Using  this new description, he shows that the poset exhibits property $S$, which is a stronger property than the Sperner property, and he provides a new proof that the poset is rank-unimodal~\cite{Sta80}.  The poset $M(n)$ is also known to be a distributive lattice, a fact that will assist us herein.

\subsection{A poset for linear games} \label{sec:jn}
Having defined a poset of coalitions, we now define posets for linear, proper, and weighted games.  

Consider two coalitions $B \ge A$; if $A$ is winning in a simple game $v$, then $B$ must also be  winning in $v$.  This means that the set $W_v$ of winning coalitions in the game $v$ extends upwards from one or more lowest winning coalitions to the top of $M(n)$.  The set $W_v$ forms what is called a \emph{filter} of $M(n)$; its lowest elements are the generators of this filter.  Since each linear game $v$ is determined by its set of winning coalitions, each one can be described uniquely as a filter of $M(n)$.   

The \emph{shift-minimal winning coalitions} (or \emph{generators}) for a linear game are the generators of the filter it represents in $M(n)$.  We observe that every shift-minimal winning coalition is minimal, but not vice-versa.  We will use the shift-minimal winning coalitions to denote a linear game; that is, if $A_1, A_2, \hdots , A_k$ are the generators for game $v$, we write $v = \langle A_1, A_2, \hdots , A_k \rangle$.  A filter possessing only one generator is known as a \emph{principal filter} and corresponds to a game with only one generating coalition.  When denoting such games, we often drop the brackets.

\begin{example}  \label{ex:induced}
Recall the weighted game $v$ on 4 players realized by $(0.6:  0.35, 0.25, 0.2, 0.2)$ introduced in Example~\ref{ex:inducedA}.  Its winning coalitions are $321, 421, 43, 431, 432, 4321$; of these, the first three are minimal.  There are two shift-minimal winning coalitions, namely $43$ and $321$.  We denote this game as $v=\langle 43, 321 \rangle$.

In Section~\ref{sec:games}, we mentioned that any linear game on $n$ players induces a similar game on $m$ players for each $m>n$; these induced games are obtained by adding dummies.  In this example, on $n=6$ players $v$ induces the game $u=\langle 65, 543 \rangle$, generated by coalitions  ${65}$ and ${543}$. Both games have power composition $(2,2)$.  Here, $u$ is also a weighted game, realized by $(0.6:  0.35, 0.25, 0.2, 0.2, 0, 0)$.
\hfill{$\Diamond$}
\end{example}

Linear games on $n$ players are in one-to-one correspondence with filters of $M(n)$.  By choosing any set of incomparable coalitions in $M(n)$, we are specifying the shift-minimal winning coalitions for some unique linear game, and we are specifying the generators of a unique filter of $M(n)$.  It is worth noting that not every filter of $M(n)$ produces a weighted game.  We may thus construct a poset of linear games using the natural partial ordering (containment) on filters as follows. 

\begin{definition} \label{def:jn}
 A linear game $v$ is \emph{stronger than} another linear game $u$, denoted $v \succeq u$, if every winning coalition in $v$ also wins in $u$, i.e., if $W_v \subset W_u$.   We refer to the set of linear games on  $n$ players under this partial ordering as the \emph{linear games poset} and denote it by $J(M(n))$ or simply $J_n$.  
\end{definition}

For any finite poset $P$, the poset $J(P)$ of filters of $P$ under the containment ordering is known to be a distributive lattice.  Thus, the linear games poset $J_n$ is a distributive lattice.  Also, $J_n$ is ranked by the number of losing coalitions in each game.  For example, the linear game $\langle 653, 5432 \rangle$ in $J_6$ has rank $48$ since there are $16$ winning coalitions in this linear game, $64$ total coalitions in $M(6)$, and therefore $64-16=48$ losing coalitions.  

\begin{remark} \label{rmk:allwin}
Formally, the linear games poset is only a \emph{subposet} of a ranked lattice, since we have excluded games of rank $0$ (where every coalition, even the empty coalition, is winning) and rank $2^n$ (where every coalition is losing) from our Definition~\ref{def:simple} of \emph{simple games}.  (A ranked poset must have minimal rank 0.)  Many authors choose to include these games.  Were we to  extend $J_n$ to include them, then the resulting poset would have minimal rank $0$ and maximal rank $2^n$.  From a voting perspective, these two games are somewhat unnatural as they represent situations where the players have no control over the outcome.  Furthermore, some of our results in Section~\ref{sec:geometry} are not valid for these two games.
\hfill{$\Diamond$}
\end{remark}

In general, determining if a linear game is weighted can be difficult.  One characterization of weighted games that will be useful below was given by Taylor and Zwicker~\cite{TZ92} in terms of general trading.  In the following, a trade is not restricted to a one-for-one exchange of players.  Any number of players can be moved among coalitions arbitrarily, provided none of the resulting coalitions contains more than one copy of any player.   A simple game $G$ is said to be \emph{trade robust} if for every collection $X=\{ X_1, X_2, \hdots , X_j \}$ of (not necessarily disjoint) winning coalitions in $G$, it is not possible to trade members among the coalitions to produce a collection $Y=\{ Y_1, Y_2, \hdots , Y_j \}$ such that the coalitions in $Y$ are all losing. 

\begin{theorem}{~\cite{TZ92}}
A game $G$ is weighted if and only if it is trade robust.
\end{theorem}

Consider the linear game on $9$ players generated by the single shift-minimal coalition $8741$, denoted $\langle 8741 \rangle$.  
Let $X=\{ 9741, 8752 \}$ and trade $8$ for $41$ to form $Y=\{ 987, 75421 \}$. The coalitions in $X$ are all winning coalitions in $\langle 8741 \rangle$ and the coalitions in $Y$ are all losing in $\langle 8741 \rangle$.  Therefore the linear game $\langle 8741 \rangle$ is unweighted since it fails to be trade robust.

\subsection{Comparing posets} \label{sec:comparing}

Let us now consider three subposets of $J_n$.  Denote by $J_n^+$ the `top half' of $J_n$, that is all games of rank exceeding or equal to half the maximal rank $2^n$.  We frequently consider proper linear games; these also form an induced subposet of $J_n$, which we denote as $\Pi_n$.  Recall that in a proper linear game, if a coalition $A$ is winning then its complement $A^c$ cannot be winning.  Therefore at most $2^{n-1}$ coalitions can be winning in a proper linear game.  This means that the elements of $\Pi_n$ must contain at least $2^{n-1}$ losing coalitions and therefore must have rank at least $2^{n-1}$; thus $\Pi_n$ lies in $J_n^+$.  However, there do exist improper linear games of rank greater than $2^{n-1}$, as we discuss in Theorem~\ref{thm:inclusions} and Appendix~\ref{app:unweighted6}.

In poset $J_n$ (and $\Pi_n$), game $v$ covers game $u$ if $W_u = W_v \cup \{A\}$ for some coalition $A \notin W_v$.   

\begin{definition} \label{defn:wn}
The \emph{weighted games poset} $\w_n$  is the set of all weighted games on $n$ players along with the partial ordering that arises from this covering relation.  We also define the poset $\w_n^+=\w_n \cap J_n^+$.
\end{definition}

In this section, we ask when the posets $\w_n^+, \Pi_n,$ and $J_n^+$ are equal.  We prove that all weighted games in $\w_n^+$ are proper, some proper games are unweighted, and some linear games in $J_n^+$ are improper.  Since $\w_n$ is defined by the same covering relation as $J_n$, it forms a subposet of the linear games poset $J_n$. We further show that $\w_n$ is an induced subposet of $J_n$.

\begin{theorem} \label{thm:inclusions}
The following inclusions hold
\begin{equation} \label{eq:wpj}
\w_n^+ \subset \Pi_n \subset J_n^+ .
\end{equation}
The first inclusion is strict precisely for $n\geq7$ players, and the second is strict precisely for $n \geq 6$.
\end{theorem}

\begin{proof}
The second inclusion, $\Pi_n \subset J_n^+$, is immediate since every proper linear game has at least half of the coalitions losing.  Thus its rank is at least $2^{n-1}$, and it lies in the top half of $J_n$.  We use the following lemma to prove the first inclusion.

\begin{lemma} \label{lem:complement-proper}
Let $v\in J_n$ be an improper weighted game.  Then $W_v$, the set of winning coalitions for $v$, must include at least one coalition from each complement pair $\{A, A^c\}$ where $A \in W_v$.
\end{lemma}

\begin{proof}[Proof of Lemma~\ref{lem:complement-proper}]
We assume $v$ is weighted and improper, so it includes a pair of complementary winning coalitions, $A$ and $A^c$.  We may trade players between these coalitions to form any desired complement pair of coalitions, $B$ and $B^c$.  Trade robustness guarantees that at least one of these coalitions is winning.
\end{proof}

Now we show the first inclusion in \eqref{eq:wpj}.  Suppose the weighted game $v$ is improper.  By the above lemma, $W_v$ includes at least one coalition from each complement pair.  As $v$ is improper, $W_v$ also includes both coalitions of some complement pair, so the rank of $v$ is less than $2^{n-1}$, i.e., $v \notin \w_n^+$.  Hence, $\w_n^+ \subset \Pi_n$.

The following examples establish when these inclusions are strict.  Note that for $n \leq 5$ players, all linear games are weighted, so both inclusions above are equalities.  For $6$ players, there are 60 unweighted games:  20 have rank less than 32, while 20 have rank equal to 32, and 20 have rank greater than 32; we list these in Appendix~\ref{app:unweighted6}.  None of them are proper.
\begin{enumerate}
\item The game $\langle 6531 \rangle \in \Pi_7$ is not weighted since it fails trade robustness - $7542$ and $6531$ are both winning but can be traded to form the two losing coalitions $765$ and $54321$.
\item Of the $40$ improper games in $J_6^+$, the one with the highest rank, 37, is $\langle 65, 4321 \rangle$.  It is generated by two complementary coalitions and is thus improper.  
\end{enumerate}
To obtain examples for larger $n$ values, simply add dummy players to these games; the induced games share the weightedness and properness of the original game.  Thus we have established the strictness criteria above.
\end{proof}

 We now show $\w_n$ is an induced subposet of $J_n$.  This question of inducement can be rephrased as follows:  if $v \succ u$ in $J_n$ for two weighted games, must there exist a saturated chain in $J_n$ from $u$ to $v$ comprised only of weighted games?   
 
\begin{proposition}  \label{prop:induced}
The weighted games poset $\w_n$ is an induced subposet of $J_n$.  
\end{proposition}

\begin{proof}
We offer a geometric proof of this combinatorial proposition.  Consider two distinct weighted games $u_1 \succ u_0$, realized by points $z_1=(q_1: \ww_1)$ and $z_0=(q_0: \ww_0)$, respectively.  We prove that there exists a saturated chain of weighted games from $u_0$ to $u_1$.  Let $\ell(\lambda)$, for $0\leq \lambda \leq 1$, be the line segment in $\C_n$ connecting these two points, with $\ell(0)=z_0$ and $\ell(1)=z_1$. Along $\ell$, consider the weighted games $u_\lambda$ realized by the quotas and weights $\left(\lambda q_1 + (1-\lambda)q_0: \lambda \ww_1 + (1-\lambda) \ww_0 \right)$.  

Consider a coalition $A$ winning in $u_0$ but losing in $u_1$.  Let $w_A$ equal the sum of the weights of the players in this coalition. Then there exists $\lambda_A \in (0,1]$ where line $\ell$ from $z_0$ to $z_1$ intersects the hyperplane $h_A$, defined by $q=w_A$.  (We utilize these hyperplanes in Section~\ref{sec:geometry} to describe the geometry of weighted games.)  Observe $A$ is winning in $u_\lambda$ if and only if $\lambda \leq \lambda_A$.  

If all $\lambda_A$ values are distinct, among all coalitions winning in $u_0$ and losing in $u_1$, i.e., among all coalitions in the set $W_{u_0} \setminus W_{u_1}$, then they provide a total ordering on this set.  From this total ordering, we obtain our desired saturated chain of weighted games from $u_0$ to $u_1$ as follows.  Delete each coalition in order from the winning coalitions for $u_0$; each deletion forms a new game until we arrive at $u_1$.  

If some $\lambda_A$ values are equal, then we may perturb the line $\ell$ to obtain a total ordering.  There are many different perturbation methods which accomplish our goal.  We now provide details here on one particularly nice geometric perturbation of $\ell$.

Suppose there are $k$ coalitions, $\{A_1, \ldots, A_k\}$ all possessing the same $\lambda_A$ value.  Let point $p=\ell(\lambda_A)$.  We can form a closed $\epsilon$-ball $B=B^n(p,\epsilon)$ centered at $p$ that is disjoint from all the other hyperplanes $h_A$, where $A\neq A_i$ for all $i$.  (Since there are only finitely many such hyperplanes, $B$ must exist for some choice of $\epsilon$.)  Let $r_0$ be the point where line $\ell$ first enters ball $B$, and $r_1$ be where $\ell$ exits $B$.  Observe all $A_i$ are winning at $r_0$ and losing at $r_1$.  

Our perturbation travels along the $(n-1)$-dimensional sphere $\Sigma$ bounding ball $B$.  On this sphere, points $r_0$ and $r_1$ are antipodal.  Between any two antipodal points on an $(n-1)$-dimensional sphere, the space of geodesics (semicircles) connecting them is canonically homeomorphic to an $(n-2)$-dimensional sphere, which may be naturally viewed as the `equator' $S^{n-2}_{eq}$, comprised of the points on the sphere $\Sigma$ equidistant from points $r_0$ and $r_1$.

\emph{Claim:}  Almost every geodesic connecting $r_0$ to $r_1$ encounters the coalitions $A_i$ at distinct points.  Equivalently, the geodesics simultaneously meeting two or more hyperplanes  form a set of measure zero on $S^{n-2}_{eq}$.

To prove the claim, consider two distinct hyperplanes $h_{A_i}$ and $h_{A_j}$.  Their intersection contains the center $p$ of ball $B$ and has dimension $n-2$.  This implies that  $h_{A_i} \cap h_{A_j}$ intersects $\Sigma$ in a great $(n-3)$-sphere\footnote{A \emph{great (sub)sphere} is one possessing the same radius as the original sphere, for example, the equator is a great circle on the earth's surface.}
 $\Sigma'$.  The set of geodesics meeting the intersection of these two hyperplanes has dimension $n-3$, and thus has measure zero in $S^{n-2}_{eq}$, which proves the claim.

The claim implies that perturbing $\ell$ by following almost any geodesic on $\Sigma$ from $r_0$ to $r_1$ will meet coalitions $A_1, \ldots, A_k$ distinctly and thereby totally order them.  Thus, even if some $\lambda_A$ values are equal, we may geometrically find a total ordering of the coalitions in $W_{u_0} \setminus W_{u_1}$, from which we produce a saturated chain of weighted games from $u_0$ to $u_1$.
\end{proof}

Just as $\w_n$ is an induced subposet of $J_n$, the proof of Proposition~\ref{prop:induced} guarantees that the poset $\w_n^+$ is an induced subposet both of $J_n^+$ and of $\Pi_n$.  Figure~\ref{fig:poset4} shows the poset $J_4^+$, which equals $\w_4^+$ and $\Pi_4$.

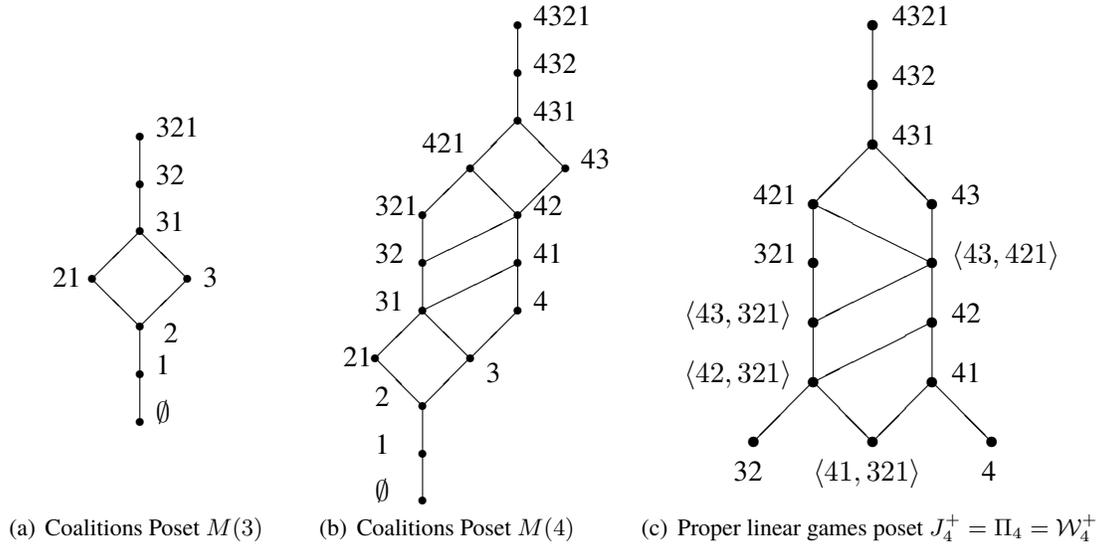
\begin{figure}[h]
\centering
\subfigure[Coalitions Poset $M(3)$]{
\setlength{\unitlength}{0.6pt}
\begin{picture}(160,250)(-80,-50)
	\put(0,180){\circle*{5}}
	\put(0,150){\circle*{5}}
	\put(0,120){\circle*{5}}
	\put(0,60){\circle*{5}}
	\put(0,30){\circle*{5}}
	\put(0,0){\circle*{5}}
	\put(30,90){\circle*{5}}
	\put(-30,90){\circle*{5}}
	\put(0,0){\line(0,1){60}}
	\put(0,120){\line(0,1){60}}
	\put(0,60){\line(1,1){30}}
	\put(0,60){\line(-1,1){30}}
	\put(0,120){\line(1,-1){30}}
	\put(0,120){\line(-1,-1){30}}
	\put(10,0){$\emptyset$}
	\put(10,30){1}
	\put(15,50){2}
	\put(-55,85){21}
	\put(40,85){3}
	\put(10,120){31}
	\put(10,150){32}
	\put(10,180){321}
\end{picture}
}
\subfigure[Coalitions Poset $M(4)$]{
\setlength{\unitlength}{0.6pt}
\begin{picture}(200,250)(-30,0)
\put(50,180){\circle*{5}}
	\put(50,150){\circle*{5}}
	\put(50,120){\circle*{5}}
	\put(50,60){\circle*{5}}
	\put(50,30){\circle*{5}}
	\put(50,0){\circle*{5}}
	\put(80,90){\circle*{5}}
	\put(20,90){\circle*{5}}
	\put(50,0){\line(0,1){60}}
	\put(50,120){\line(0,1){60}}
	\put(50,60){\line(1,1){30}}
	\put(50,60){\line(-1,1){30}}
	\put(50,120){\line(1,-1){30}}
	\put(50,120){\line(-1,-1){30}}
	\put(20,0){$\emptyset$}
	\put(20,30){1}
	\put(20,60){2}
	\put(0,85){21}
	\put(90,75){3}
	\put(20,120){31}
	\put(20,150){32}
	\put(20,180){321}
	\put(110,300){\circle*{5}}
	\put(110,270){\circle*{5}}
	\put(110,240){\circle*{5}}
	\put(110,180){\circle*{5}}
	\put(110,150){\circle*{5}}
	\put(110,120){\circle*{5}}
	\put(140,210){\circle*{5}}
	\put(80,210){\circle*{5}}
	\put(110,120){\line(0,1){60}}
	\put(110,240){\line(0,1){60}}
	\put(110,180){\line(1,1){30}}
	\put(110,180){\line(-1,1){30}}
	\put(110,240){\line(1,-1){30}}
	\put(110,240){\line(-1,-1){30}}
	\put(120,120){4}
	\put(120,150){41}
	\put(120,180){42}
	\put(50,220){421}
	\put(150,210){43}
	\put(120,240){431}
	\put(120,270){432}
	\put(120,300){4321}
	\put(80,90){\line(1,1){30}}
	\put(50,180){\line(1,1){30}}
	\put(50,120){\line(2,1){60}}
	\put(50,150){\line(2,1){60}}
\end{picture}}
\subfigure[Proper linear games poset $J_4^+ = \Pi_4 = \w_4^+$]{
\setlength{\unitlength}{0.75pt}
\parbox[b]{2.5in}{
\begin{picture}(200,250)(-90,-60)
	\put(30,180){\circle*{5}}
	\put(30,150){\circle*{5}}
	\put(30,120){\circle*{5}}
	\put(60,90){\circle*{5}}
	\put(0,90){\circle*{5}}
	\put(60,60){\circle*{5}}
	\put(0,60){\circle*{5}}
	\put(60,30){\circle*{5}}
	\put(0,30){\circle*{5}}
	\put(60,0){\circle*{5}}
	\put(0,0){\circle*{5}}
	\put(30,-30){\circle*{5}}
	\put(-30,-30){\circle*{5}}
	\put(90,-30){\circle*{5}}
	\put(60,0){\line(0,1){90}}
	\put(0,0){\line(0,1){90}}
	\put(30,120){\line(0,1){60}}
	\put(60,60){\line(-2,1){60}}
	\put(60,60){\line(-2,-1){60}}
	\put(60,30){\line(-2,-1){60}}
	\put(60,0){\line(-1,-1){30}}
	\put(60,0){\line(1,-1){30}}
	\put(0,0){\line(-1,-1){30}}
	\put(0,0){\line(1,-1){30}}
	\put(30,120){\line(-1,-1){30}}
	\put(30,120){\line(1,-1){30}}
	\put(70,0){41}
	\put(70,30){42}
	\put(70,90){43}
	\put(85,-50){4}
	\put(-40,-50){32}
	\put(-65,30){$\langle 43, 321\rangle$}
	\put(-65,0){$\langle 42, 321 \rangle$}
	\put(-30,60){321}
	\put(-30,90){421}
	\put(70,60){$\langle 43, 421 \rangle$}
	\put(40,120){431}
	\put(40,150){432}
	\put(40,180){4321}
	\put(0,-50){$\langle 41, 321\rangle $}
\end{picture}}
}
\caption{Poset examples for 3 and 4 players}
\label{fig:poset4}
\end{figure}

\subsection{Poset properties} \label{sec:properties}
We now consider various properties of our posets of games, including their ranks, covers, and inclusions. 

Recall that the filters in the poset $J_n$ are ranked by the cardinality of their losing coalitions, and since the cardinalities vary from 1 to $2^n-1$, the poset $J_n$ has rank $2^n-1$.

\begin{proposition} 
The weighted games poset $\w_n$ achieves each rank from $1$ to $2^n -1$.  The posets $\w_n^+, \Pi_n,$ and $J_n^+$ achieve each rank from $2^{n-1}$ to $2^n -1$.
\end{proposition}

\begin{proof}
By construction $\w_n \subset J_n$; Theorem~\ref{thm:inclusions} showed that $\w_n^+ \subset \Pi_n \subset J_n^+$.  So, it suffices to construct a weighted game for each rank.

Let player $\v{i}$ have unnormalized weight $2^{i-1}$.  Under these weights, the coalitions are totally ordered in the sense that no two coalitions have the same weight.  Choosing a quota of 1 produces the weighted game $\langle 1 \rangle$ which has rank 1; choosing a quota of 3 produces the weighted game $\langle 21 \rangle$ which has rank 3.  In general, choosing a quota of $r$ produces a weighted game of rank $r$.  Quotas of $0$ and $2^n$ correspond to the situations where all coalitions and no coalitions, respectively, are winning; we exclude these cases from our definition of simple games (cf.\  Remark~\ref{rmk:allwin}). 
\end{proof}

\begin{corollary} {\label{cor:complements}}
In the proper games poset $\Pi_n$, the games with minimal rank $2^{n-1}$ are precisely the self-dual linear games.
\end{corollary}

\begin{proof}
Recall that a game $v$ is self-dual if and only if it is both proper and strong, or equivalently, for each winning coalition $A$ in $v$, its complement $A^c$ must be losing in $v$.  Therefore every self-dual game has rank $2^{n-1}$ and is proper.

Conversely, let $v$ be a proper linear game with rank $2^{n-1}$ in $\Pi_n$.  Then precisely one coalition from each complement pair $\{A, A^c\}$ is winning in $v$, since at most one of $A$ and $A^c$ can be winning in $v$.  Therefore $v$ is self-dual.
\end{proof}

\begin{remark}
A natural question arises:  are all linear games with `middle' rank $2^{n-1}$ self-dual?  For $n \leq 5$ players, the answer is yes, since the posets $\Pi_n$ and $J_n^+$ coincide.  For $n \geq 6$ players however, the answer becomes no, since $\Pi_n$ is strictly contained in $J_n^+$.  We observe that in $J_6$ there exist 41 games of `middle' rank $32$; only 21 are proper -- these are the self-dual ones. The other 20, listed in Appendix~\ref{app:unweighted6}, are improper and occur in dual pairs.  
\hfill{$\Diamond$}
\end{remark}

\begin{proposition} \label{prop:covers}
A linear game with $k$ shift-minimal winning coalitions is covered by precisely $k$ elements in $J_n$.  If the game is proper, the statement holds also in $\Pi_n$.
\end{proposition}

\begin{proof}
Let $v$ be a linear game with $k$ shift-minimal winning coalitions.  The linear games covering $v$ in $J_n$ are the filters of $M(n)$ obtained by removing exactly one of the shift-minimal winning coalitions from $W_v$.  Since there are $k$ generators which can be removed, there are $k$ different filters covering $v$ in $J_n$.

If $v$ is proper, removing a winning coalition will retain properness, so proper games are only covered by proper games.  Thus our result holds in $\Pi_n$ as well.
\end{proof}

We observe that Proposition~\ref{prop:covers} is not true for weighted games.  For example, the game $\langle 987,8741 \rangle$ in $\w_9^+$ is weighted with two shift-minimal winning coalitions: $987$ and $8741$.  It is covered only by $\langle 987, 9741,8751,8742 \rangle$ and not by $\langle 8741 \rangle$ since $\langle 8741 \rangle$ is not weighted.

We now describe a useful inclusion of $J_n$ into $J_{n+1}$.  Recall from the end of Section~\ref{sec:background} that the game $v = \langle A_1, A_2, \ldots \rangle$ on $n$ players induces the  game $\tilde{v}$ by adding a dummy player.  So we have a map $J_n \hookrightarrow J_{n+1}$ that sends $v$ to $\tilde{v}$.  The rank of $\tilde{v}$ is twice that of $v$.  Thus, we may conclude that if a game has $k$ dummies, its rank must be a multiple of $2^k$.
Since induced games preserve weightedness and properness, this map also produces the inclusions $\w_n \hookrightarrow \w_{n+1}$ and $\Pi_n \hookrightarrow \Pi_{n+1}$.

\subsection{Enumerating linear games}

The tasks of counting linear and weighted games are difficult since the number of each grows rapidly as the number of players increases; full results are known only for $n \leq 9$ players~\cite{KS95, MTB70}.  The enumeration of simple games has been studied by mathematicians for over a century, beginning with Dedekind's 1897 work in which he determined the number of simple games with four or fewer players.  
 Recently Freixas and Puente~\cite{FrePue08} investigated linear games with one shift-minimal winning coalition, Kurz and Tautenhahn~\cite{KurTau11} have enumerated linear games with two shift-minimal winning coalitions, and Freixas and Kurz~\cite{FK11} have provided a formula for the number of weighted games with one shift-minimal winning coalition and two types of players.  We extend this research to proper linear games by counting those with one shift-minimal winning coalition.

Posets provide a natural tool for ensuring that we have enumerated all linear games for $n$ players.  For a linear game $v$, its set of winning coalitions $W_v$ is the filter in $M(n)$ generated by the shift-minimal winning coalitions in $v$.  A new linear game can be obtained by either removing a shift-minimal winning coalition from $W_v$ or adding a new coalition to $W_v$ which is covered only by elements of $W_v$ but is not in $W_v$, i.e., a \emph{shift-maximal losing coalition} of $v$.  This procedure either adds one 
(or subtracts one, respectively) to the rank (resp. from the rank) of the linear game.  To obtain a new linear game with the same rank as $v$, perform both operations: remove a winning coalition $A$ from $W_v$ and add a new winning coalition which is not in $W_v \setminus \{A\}$ but is covered only by elements of $W_v \setminus \{A\}$.  Note that it may not always be possible to perform both of these operations.

\begin{theorem}
For $n$ players, the number of proper linear games generated by exactly one shift-minimal winning coalition is 
\begin{equation} \label{eq:1gen}
2^n - {n \choose \lfloor n/2 \rfloor}.
\end{equation}
The coalitions $A$ for which a prinipal filter, $\langle A \rangle$, of $M(n)$ is a proper game are precisely the subsets of $[n]$   which contain $k$ of the largest $2k-1$ numbers in $[ n ]$ for some $k \le n$.
\end{theorem}

\begin{proof}
We consider the $2^n$ different games $\langle A \rangle$, where $A \subset [n]$.  Such a game is proper if and only if $A^c \ngtr A$, i.e., if the complement of $A$ does not lie in the filter $\langle A \rangle$.  By Definition~\ref{partialorder}, $A^c \ngtr A$ is equivalent to having the $k$th element of $A$ be greater than the $k$th element of $A^c$ for some $k \leq n$.  Thus, $\langle A \rangle$ is proper if and only if $A$ contains $k$ of the largest $2k-1$ numbers in $[n]$ (for some $k\leq n$).  

This is equivalent to the number of ways to flip a fair coin $n$ times so that a majority of heads had occurred at some point, which is enumerated as sequence A045621 in the Online Encyclopedia of Integer Sequences~\cite{OEIS} and given by formula \eqref{eq:1gen} above.
\end{proof}


\section{The geometry of weighted voting realizations} \label{sec:geometry}

We now study the geometry of realizations of weighted games.  Recall from Section~\ref{sec:background} that since weighted voting is scale invariant, we normalize the weights so that they sum to 1.  Also, $\Delta_n$ denotes the $(n-1)$-dimensional simplex of normalized weights for weighted, $n$-player games and $\C_n = (0,1] \times \Delta_n$ denotes the space of all realizations of such games.  We envision $\C_n \subset \R^n$ depicted with coordinate $q$ pointing upwards (in the vertical direction) and will refer to `top' and `bottom' features based on appropriate $q$ values.

Consider all of the realizations  in $\C_n$ for a weighted game $v \in \w_n$; these points define a polytope $P_v$ in $\C_n$.  The polytopes $P_v$ encode a rich amount of information about weighted games; the goal of this section is to describe the geometry of weighted voting and its connections with posets and hierarchies.

By a \emph{polytope}, we mean the generalization of polygons or polyhedra to bounded $k$-dimensional objects.  In $\R^k$, each polytope is bounded by a finite number of hyperplanes; these define $(k-1)$-dimensional subpolytopes called \emph{facets}.  We do not assume that all polytopes are convex; a convex polytope may be viewed as the convex hull of a finite set of points.  Each $P_v$ is in fact convex, as we demonstrate in Proposition~\ref{prop:convex}.  The polytopes $P_v$ are neither open nor closed; they contain some but not all of their facets, as we will discuss later in this section.

Let us begin by examining the geometry of the configuration regions:  $\C_1$ is a line segment of quotas above the point $p_1$ (where $w_1=1$) while $\C_2$ is the rectangle $(0,1] \times  \overline{p_2 p_1}$.  Notice that $\C_1$ embeds naturally into $\C_2$.  Figure~\ref{fig:config3} depicts $\C_3$, which is a triangular prism.  Notice that  $\C_2$ embeds into $\C_3$ as the back facet, and $\C_1$ embeds as the rightmost edge.  This is true in general:  every $\C_k$ naturally embeds into $\C_n$ for $k<n$.  

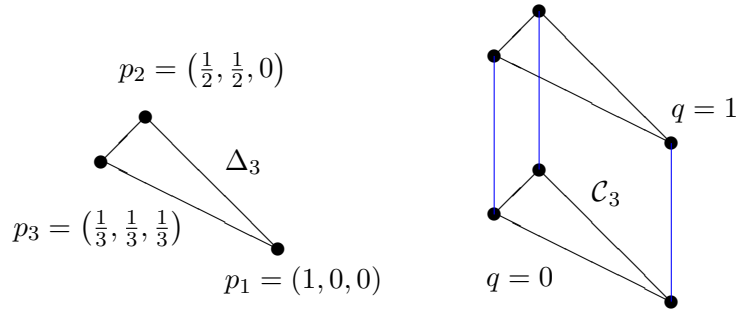
\begin{figure} [ht]
\begin{picture}(100,100)(0,0)
	\put(100,0){\line(-1,1){50}}
	\put(100,0){\line(-2,1){67}}
	\put(50,50){\line(-1,-1){17}}
	\put(33,33){\circle*{5}}
	\put(50,50){\circle*{5}}
	\put(100,0){\circle*{5}}
	\put(80,30){$\Delta_3$}
	\put(80,-15){$p_1=(1,0,0)$}
	\put(0,5){$p_3=\left(\tfrac{1}{3}, \tfrac{1}{3}, \tfrac{1}{3}\right)$}
	\put(40,65){$p_2=\left( \tfrac{1}{2}, \tfrac{1}{2}, 0\right)$}
\end{picture} 
\hspace*{0.6in}
\begin{picture}(100,100)(0,20)
	\put(100,0){\line(-1,1){50}}
	\put(100,0){\line(-2,1){67}}
	\put(50,50){\line(-1,-1){17}}
	\put(33,33){\circle*{5}}
	\put(50,50){\circle*{5}}
	\put(100,0){\circle*{5}}
	\put(100,60){\line(-1,1){50}}
	\put(100,60){\line(-2,1){67}}
	\put(50,110){\line(-1,-1){17}}
	\put(33,93){\circle*{5}}
	\put(50,110){\circle*{5}}
	\put(100,60){\circle*{5}}
{\color{blue}	\put(100,0){\line(0,1){60}}
			\put(50,50){\line(0,1){60}}
			\put(33,33){\line(0,1){60}}}
	\put(30,6){$q=0$}
	\put(100,70){$q=1$}	
	\put(70, 38){$\C_3$}
\end{picture} 
\vspace*{0.2in}
\caption{For 3 players, $\Delta_3$ is the region of normalized weights, and $\C_3$ is the configuration region of quotas and weights.}
\label{fig:config3}
\end{figure}

For a weighted game $u$ on $m$ players induced from a game $v$ on $n<m$ players, the polytope $P_v$ is the projection of $P_u$ under the natural projection $\C_m \rightarrow \C_n$.  Thus the geometry of $\C_m$ completely determines the geometry of $\C_n$.  

\subsection{Polytope structure} \label{sec:polytope}
Let us now describe how the polytopes in the configuration region $\C_n$ are formed.  First, let $\w_A$ equal the sum of the weights of players in coalition $A$.  Consider the set of points in $\C_n$ where $q=w_A$, i.e., where $A$ has precisely enough weight to win; the set of these points lies in a hyperplane $h_A$.  Unless $A$ is empty or equal to the grand coalition $[n]$, the hyperplane $h_A$ intersects $\C_n$ in a codimension one subset that slants -- its normal vector is neither horizontal nor vertical.  Observe that $h_{[n]}$ forms the top facet of $\C_n$ and $h_\emptyset$ the bottom  facet; the latter is not actually contained in $\C_n$.  

\begin{remark}
These hyperplanes respect the ordering on coalitions in $M(n)$; we have $A > B$ if and only if $h_A$ lies strictly above $h_B$ on the interior of $\C_n$.  Equivalently, $A$ and $B$ are incomparable coalitions if and only if their hyperplanes intersect on the interior of $\C_n$.  \hfill{$\Diamond$}
\end{remark}

Coalition $A$ is winning at a realization $(q:\ww)$ if and only if $q\leq w_A$, so we may visualize the points which have $A$ winning as the closed subset $X_A$ of $\C_n$ bounded above by hyperplane $h_A$.  Similarly the points which have coalition $B$ losing form the open subset $(X_B)^c =\C_n \setminus X_B$ which is bounded below by $h_B$.  Thus, for a weighted game $v$, we may view its polytope $P_v$ as the intersection of all `winning subsets' such as $X_A$ and all `losing subsets' such as $(X_B)^c$:
$$P_v = \left( \bigcap_{A \in W_v} X_A \right) \bigcap \left( \bigcap_{B \notin W_v} (X_B)^c \right).$$  

This formulation demonstrates that each polytope $P_v$ is closed on top and open on bottom and is convex.  
\begin{proposition} \label{prop:convex}
Each polytope $P_v$ associated to a weighted game $v$ is convex.
\end{proposition}

We will be interested in what occurs by moving along a vertical line in $\C_n$ from a realization $z=(q: \ww)$ in $P_v$.  These motions are equivalent to changing the quota while fixing the weights.
\begin{itemize}
\item Moving upwards from $z$ (increasing the quota) guarantees that each losing coalition will remain losing.  A coalition $A$ winning at $z$ remains winning until after the line crosses the hyperplane $h_A$.

\item Moving downwards from $z$ (decreasing the quota) guarantees that each winning coalition will remain winning.  A coalition $B$ losing at $z$ remains losing until the line intersects the hyperplane $h_B$.
\end{itemize}

\begin{proposition} \label{ndim}
Let $v$ be a weighted game on $n$ players.  Each polytope $P_v$ is $n$-dimensional; that is, it has full dimension in the configuration region $\C_n$.
\end{proposition}

\begin{proof}
Since $v$ is weighted, there exists some representation $z_0=(q_0: \ww)$ for the game, so each polytope $P_v$ includes at least one point.  Move `upwards' by fixing the weight vector $\ww$ and increasing the quota until reaching the top boundary of $P_v$ at some point $z=(q_1: \ww)$; 
n.b., $z$ might equal $z_0$.

We first show that $z$ itself lies in the polytope $P_v$.  Points in polytope $P_v$ all satisfy the same inequalities:  $q \leq w_A$ for any winning coalition $A$ and $q > w_B$ for any losing coalition $B$.  Moving upwards from $z_0$, we first encounter the boundary of $P_v$ at the lowest point where one or more of the winning inequalities becomes an equality $q=w_A$.  The point $z$ satisfies the same inequalities as $z_0$ does - if $z_0 \neq z$, all inequalities at $z_0$ are strict, whereas if $z_0 = z$ then at least one at $z$ is weakly satisfied.  Thus, $z$ is a realization of $v$.

Assume at $z$ there are $k$ inequalities that are weakly satisfied, corresponding to coalitions $A_1, \ldots, A_k$, with weight equal to quota $q_1$.  The remaining $2^n - k$ hyperplanes lie either above or below point $z$; let $\delta$ be the minimum distance down to the next highest hyperplane(s) and let $r=(q_1-\delta : \ww)$.   As we travel downwards from $z$, no coalition will ever change from winning to losing; only when we encounter the next hyperplane does some coalition(s) change from losing to winning.  Hence, all points moving down from $z$ are in $P_v$ until we reach a quota of $q_1-\delta$ at $r$. 

Recall that we are using coordinates $\{w_n, \ldots, w_2\}$ on $\Delta_n$ with $w_1 = 1 - w_n -\cdots - w_2$.  Writing out the hyperplane equation $q=w_A$, we see that its slope  in direction $w_i$, for $i>1$, is one of $\{+1, 0, -1\}$:
\begin{itemize}
\item $+1$ if $\v{i}\in A$, but $1 \notin A$,
\item $-1$ if $i \notin A$, but $1 \in A$,
\item $0$ if both $\v{i}$ and $1$ are in $A$ or neither is in $A$.
\end{itemize}
Thus we know that the interior of the diamond depicted below formed by points $z, r$, and $ z_i^\pm = (q_1-\frac{\delta}{2} : w_n, \ldots, w_i \pm \frac{\delta}{2}, w_{i-1}, \ldots)$ lies in $P_v$.  
\begin{center}
\begin{picture}(50,50)
\put(25,5){\circle{5}}
\put(25,45){\circle*{5}}
\put(5,25){\circle{5}}
\put(45,25){\circle{5}}
\put(27,7){\line(1,1){16}}
\put(23,7){\line(-1,1){16}}
\put(25,45){\line(1,-1){18}}
\put(25,45){\line(-1,-1){18}}
\put(30,3){{\scriptsize $r$}}
\put(26,50){{\scriptsize $z$}}
\put(-10,26){{\scriptsize $z_i^-$}}
\put(50,26){{\scriptsize $z_i^+$}}
\put(38,0){\vector(1,0){25}}
\put(48,3){{\scriptsize $w_i$}}
\end{picture}
\end{center}
Thus, the convexity of $P_v$ implies that  it contains the interior of the $n$-dimensional polytope  (a cross-polytope) spanned by points $z, r, z_2^\pm, \ldots, z_n^\pm$.  So we have shown $P_v$ is $n$-dimensional.
\end{proof}

The boundary of polytope $P_v$ is comprised of three different types of facets.  We will count these in Theorem~\ref{thm:facets} using the poset $\w_n$ and the hierarchy of $v$.
\begin{enumerate}
\item \emph{top facets} -- each is associated to a hyperplane $h_A$ for some winning coalition $A$; the top facet's interior is contained in $P_v$.

\item \emph{bottom facets} -- each is associated to a hyperplane $h_A$ for some losing coalition $A$; a bottom facet is disjoint from $P_v$;

\item \emph{vertical facets} -- each lies above a codimension one subsimplex of $\Delta_n$; the interior of a vertical facet is contained in $P_v$.  These are where the polytope intersects the boundary of the configuration space.
\end{enumerate}

\begin{example}{\label{ex:31}}
The game $\langle 31 \rangle$ in Figure~\ref{fig:p1p3}, for example, has top facet $h_{31}$, bottom facets $h_{21}$ and $h_3$, and vertical facet contained in $\overline{p_1 p_3}$. 
\hfill{$\Diamond$}
\end{example}

Higher codimension elements of $P_v$ are formed by the intersection of two or more facets.  They are included in $P_v$ if and only if no bottom facets are part of the intersection.  

Now we turn our attention back to dual games.  A game and its dual share many properties, including power compositions, weightedness, and congruent polytope interiors.

\begin{theorem} \label{thm:dualgeom}
The interior of $P_v$, the polytope associated to the weighted game $v$, is the reflection in $\C_n$ about the hyperplane $q=0.5$ of the interior of the polytope $P_{v^*}$ associated to its dual game $v^*$.
\end{theorem}

\begin{proof}
Let $(q: \ww)$ be a point in the interior of $P_v$.  We prove that $(1-q: \ww)$ lies in the interior of $P_{v^*}$.  Consider a coalition $A$ for game $v$.  Then $A$ is winning in $v$ if and only if $A^c$ is losing in $v^*$, so the weight of $A$ satisfies  $w_A > q$ if and only if $w_{A^c} < 1-q$.  (This establishes the well-known fact that $v^*$ is weighted if and only if $v$ is.)
\end{proof}

As an immediate corollary, we gain another characterization of self-dual weighted games, which we showed in Corollary~\ref{cor:complements} are precisely the ones that lie in the middle rank $2^{n-1}$ in $\w_n$.
\begin{corollary}
The weighted game $v\in \w_n$ is self-dual (i.e., $v=v^*$) if and only if its polytope $P_v$ is symmetric about the hyperplane $q=0.5$.
\end{corollary}

\subsection{Polytopes and hierarchies}
In this section, we describe how the hierarchy of players in a weighted game $v$ is related to the polytope $P_v$.  We begin with several ideas which will be useful in future proofs and then specify a one-to-one correspondence between $(k-1)$-dimensional subsimplices of $\Delta_n$ and compositions into $k$ parts of natural numbers less than or equal to $n$.  

\begin{definition}
Let $\pi(P_v)$ be the vertical projection of the polytope $P_v$ onto $\Delta_n$; we call $\pi(P_v)$ the \emph{footprint} of $v$.  
\end{definition}

The footprint is comprised of all weights that are part of some realization of $v$.

\begin{lemma}{\label{EqualWeights}}
There exists a realization of $v$ in which all players in the same symmetry class have the same weight.
\end{lemma}

\begin{proof}
Given a realization $(q: \ww)$ of $v$, our strategy is to replace the weight $w_i$ by the average of the weights of all players in the same symmetry class $[i]$ as player $\v{i}$.  We claim this operation preserves the set of winning coalitions $W_v$.  Assuming the claim, we perform this operation for each symmetry class and thereby constructively prove the lemma.

Consider a coalition $B \in W_v$ which contains $m$ players from class $[i]$.  By the definition of player symmetry, $B$ will remain winning if we replace these players by the weakest $m$ players in class $[i]$.  Furthermore, $B$ will remain winning if we reassign weights to the players in class $[i]$, so long as the new weights do not cause the total weight of the weakest $m$ players in $[i]$ to decrease.  The new weights must also respect the ordering on players.  Replacing each weight by the average accomplishes both conditions, and thus causes $B$ to remain winning.

A similar argument shows that losing coalition $C$ will remain losing if we average weights within a symmetry class.  Thus, $W_v$ is preserved by the averaging operation and the claim is proven.
\end{proof}

Note that we can average the weights of any two (or more) consecutive players in the same symmetry class and still obtain a realization of $v$ using the same quota.

Let $\sigma$ be a $(k-1)$-dimensional subsimplex whose vertices are $p_{i_1},  p_{i_2},  \ldots, p_{i_k}$, where $i_1 < i_2 < \hdots < i_k$.  Recall that $p_i = \left(\tfrac{1}{i}, \ldots, \tfrac{1}{i}, 0, \ldots, 0 \right)\in \Delta_n$ has $i$ nonzero coordinates.  
This implies that for any point $\ww$ in $\sigma$, the first $i_1$ coordinates of ${\bf w}$ are equal, as are the next $i_2-i_1$ coordinates, and so forth.
Thus, the coordinates of any point $\ww$ in $\sigma$ can be grouped into $k$ different classes of equal values; if $i_k \neq n$, there is one additional class of $n-i_k:=m_0$ coordinates, which are all $0$.  Let $m_1 = i_1$ and $m_j=i_j-i_{j-1}$ for $j>1$.  We refer to the composition $\ca_{\sigma}:=(m_1,m_2, \hdots , m_k)$ as the \emph{composition associated to $\sigma$}.  This establishes a bijection between $(k-1)$-dimensional subsimplices of $\Delta_n$ and compositions of $n-m_0$ into $k$ parts.  By appending $m_0$ to $\ca_{\sigma}$, we form the \emph{extended composition associated to $\sigma$}, denoted $\eca_{\sigma}$.

We observe that simplex $\sigma \subset \Delta_n$ contains simplex $\tau \subset \Delta_n$ if and only if composition $\gamma_{ \sigma}'$ refines $\gamma'_{ \tau}$.

The following theorem establishes a relation between polytopes and power compositions.  

\begin{theorem} \label{thm:hierarchy}
For a weighted game $v$, let $\sigma_v \subset \Delta_n$ be the smallest dimensional subsimplex that intersects the footprint of $v$.  The power composition of $v$ is the composition associated to $\sigma_v$.
\end{theorem}

The theorem implies that the power composition of a weighted game $v$ can be obtained directly from its polytope $P_v$.   While the converse is untrue, the power composition does tell us precisely which subsimplices of $\Delta_n$ intersect the polytope, namely those which contain $\sigma_v$.  

\begin{example}
Consider the weighted game $v=\langle 31 \rangle$ from Example~\ref{ex:31}. The vertices of the polytope $P_v$ are $(1: 1,0,0),\left(\frac{2}{3} : \frac{1}{3}, \frac{1}{3}, \frac{1}{3}\right), \left(\frac{1}{2} : \frac{1}{2}, \frac{1}{4}, \frac{1}{4}\right),$ and $\left(\frac{1}{2} : \frac{1}{2}, \frac{1}{2}, 0 \right)$.  Its footprint is the interior of the triangle with vertices $p_1, p_2$ and $p_3$, together with the segment from $p_1$ to $p_3$.  Therefore $\sigma_v= \overline{p_1 p_3}$ and thus by Theorem~\ref{thm:hierarchy} the power composition of $v$ is $(1,2)$.
\hfill{$\Diamond$}
\end{example}

To prove Theorem~\ref{thm:hierarchy}, we first need a definition and a lemma.  We obtain the \emph{extended power composition} $\epc (v)= (n_1, n_2, \hdots , n_k, n_o)$ of $n$ for $v$ by appending $n_0$, the number of dummies in $v$, to the power composition of $v$.  

\begin{lemma}{\label{IntRes}}
Let $v$ be a weighted game.  A subsimplex $\tau \subset \Delta_n$ intersects the footprint of $v$  if and only if the extended composition ($\eca_{\tau}$) associated to $\tau$ is a refinement of the extended power composition $\epc(v)$.  Furthermore, if $\tau$ intersects the footprint, then the horizontal projection of polytope $P_v$ onto the element $\epsilon_\tau:= (0,1] \times \tau \subset \C_n$ is contained in $P_v$, i.e., it equals $P_v \cap \epsilon_\tau$.
\end{lemma}

\begin{proof}
Let $v$ be a weighted game on $n$ players and let $\tau$ be an arbitrary subsimplex of  $\Delta_n$ (of any dimension) given by vertices $p_{i_1},p_{i_2}, \hdots , p_{i_j}$. 
We first assume that $\tau$ intersects $\pi(P_v)$ and show that the extended composition $\eca_{\tau}$ associated to $\tau$ refines $\epc(v)$.  If $\tau$ intersects $\pi(P_v)$, then there is a realization $(q: \ww) \in P_v$ such that $\ww \in \tau$.  In this realization, the first $m_1:=i_1$ players have the same weights, the next $m_2:=i_2-i_1$ players have the same weights, and so forth.  Players with equal weights must lie in the same symmetry class.  This means the first $m_1$ players must lie in the same symmetry class in $v$, the next $m_2$ players must lie in the same symmetry class, which is possibly the same symmetry class as the first $m_1$ players, and so forth.  Therefore the extended composition $\eca_{\tau}=(m_1, m_2, \hdots m_j,m_0)$ associated to $\tau$ refines the extended power composition.

Now we prove the reverse implication.  We assume $\eca_{\tau}$ refines $\epc(v)$ and show that $\tau$ must intersect $\pi(P_v)$.  First, there exists a realization $z=(q:\ww)$ of $v$ for which all players in the same symmetry class have the same weight by Lemma~\ref{EqualWeights}.  Since the polytope is $n$-dimensional, we may travel a short distance away from $z$ (remaining inside $P_v$) along some vector which keeps the weights of players in the same part of $\eca_{\tau}$ the same but ensures that the weights of players in different parts of $\eca_{\tau}$ are different.  We arrive at a point $z' \in P_v$ lying above $\tau$.  Hence $\tau$ intersects the footprint of $v$, which establishes the first statement of the lemma.

Now we prove the second statement.  From the proof of Lemma~\ref{EqualWeights}, we concluded that if we replace the weights of any number of consecutive players in the same class by their average, we stay within $P_v$.  Since the quota remains fixed, this operation corresponds to a horizontal motion within polytope $P_v$ from an arbitrary realization of $v$ to a point on some element $\epsilon_\tau$ in the boundary of $\C_n$.   Thus the horizontal projection of $P_v$ onto $\epsilon_\tau$ is already inside $P_v$.
\end{proof}

\begin{proof}[Proof of Theorem~\ref{thm:hierarchy}.]
We first prove that the statement of the theorem is well-defined, that is, that there exists a unique smallest-dimensional subsimplex $\sigma \subset \Delta_n$ which intersects the footprint of $v$.  Consider two subsimplices $\sigma_1$ and $\sigma_2$ that both intersect $\pi(P_v)$.  The vertices $p_i$ which lie in both $\sigma_1$ and $\sigma_2$ determine the subsimplex $\sigma_1 \cap \sigma_2$.  Note that this intersection is necessarily nonempty, else all players become dummies since Lemma~\ref{IntRes} implies that the players' symmetry classes in $v$ are refined by both $\eca_{\sigma_1}$ and $\eca_{\sigma_2}$.  The composition associated to $\sigma_1 \cap \sigma_2$ is the common refinement of the compositions associated to $\sigma_1$ and $\sigma_2$ and respects the symmetry classes of the players in $v$.  By Lemma~\ref{IntRes}, $\sigma_1 \cap \sigma_2$ must intersect $\pi(P_v)$ as well.  Thus, there must exist a unique smallest subsimplex $\sigma$ which intersects $\pi(P_v)$.  Furthermore, $\sigma$ is contained in every subsimplex that intersects $\pi(P_v)$.

From Lemma~\ref{IntRes} we know that the extended composition $\eca_{\sigma}$ associated to $\sigma$ refines the extended power composition $\epc(v)$ of $v$.  We will prove that $\ca_{\sigma}= \gamma_v$.  Note that $\epc(v)$ is associated to some subsimplex $\tau$.  Lemma~\ref{IntRes} implies that $\tau$ intersects $P_v$.  Thus $\sigma \subseteq \tau$, which implies that the vertices of $\sigma$ are a subset of the vertices of $\tau$.  Thus, the extended composition $\epc(\tau)$ associated to $\tau$ refines the extended composition $\eca(\sigma)$ associated to $\sigma$.  Therefore $\epc(v)$ refines $\eca_{\sigma}$ and $\eca_{\sigma}$ refines $\epc(v)$.  So $\eca_{\sigma} = \epc(v)$ and the proof is complete.
\end{proof}

Having witnessed how geometry and power are intertwined, we next examine what happens at the extremes.  Geometrically, what does it mean for all voters have distinct powers?  Or to have the same power?  The two corollaries below respectively answer these questions.   The former requires the polytope to entirely avoid the boundary of $\C_n$ while the latter only occurs if the polytope lies above a vertex of the simplex $\Delta_n$ of normalized weights.

The first corollary describes what must occur in order for each player to have a distinct amount of power, i.e., each player lying in its own, distinct symmetry class.  This situation does not occur for games on $n \leq 4$ players.  Among all games on $n=5$ players, exactly one game imparts distinct powers to its players, namely the game $\langle 54, 531, 432 \rangle$.

\begin{corollary}
For a weighted game $v$, the following are equivalent.
\begin{enumerate}
\item Each player lies in its own, distinct symmetry class.
\item The power composition of $v$ is $(1^n)=(1,1, \hdots , 1)$.
\item The polytope $P_v$ is contained in the interior of $\C_n$.
\end{enumerate}
\end{corollary}

At the other extreme, in a \emph{symmetric game} (or \emph{collegium}), each nondummy player lies in the same symmetry class.   The winning coalitions of a symmetric game are precisely those containing at least $k$ out of the $n-n_0$ nondummy players; its power composition is $(n-n_0)$.  Each symmetric game with no dummies is denoted $\langle (k)(k-1) \cdots 21\rangle$, where $k$ is the number of players needed for a coalition to win.   There are $n$ symmetric games ($k=1, \ldots, n$) on $n$ players with no dummies, $n-1$ symmetric games with 1 dummy, $n-2$ symmetric games with 2 dummies, and so forth.   This sums to ${n+1 \choose 2}$ symmetric games on $n$ players; all symmetric games are weighted.  Of these, $\frac{n^2+2n}{4}$ are proper if $n$ is even, and $\frac{n^2+2n+1}{4}$ are proper if $n$ is odd.  

\begin{corollary}For a weighted game $v$, the following are equivalent.
\begin{enumerate}
\item The game is symmetric.
\item Each nondummy player lies the same symmetry class.
\item The power composition of $v$ has one part.  
\item The polytope $P_v$ lies above one vertex of $\Delta_n$
\end{enumerate}
\end{corollary}

Above point $p_j$ lie $j$ different symmetric games; each occupies quotas of length $1/j$ above $p_j$.

\begin{example} \label{ex:n=3}
Let us consider the 8 weighted games for $n=3$ players.  Of these, 6 are symmetric games.  Games $\langle 1 \rangle, \langle 21 \rangle, \langle 321 \rangle$ each have power composition $(3)$.  These are the only games which lie above the $0$-dimensional subsimplex $\{p_3\}$.  Representations $(q: p_3)$ lie in $\langle 1 \rangle$ for $q \in (0, 1/3]$, lie in $\langle 21 \rangle$ for $q \in (1/3, 2/3]$, and lie in $\langle 321 \rangle$ for $q \in (2/3, 1]$.

Games $\langle 2 \rangle$ and $\langle 32 \rangle$ each have power composition $(2)$.  These are the only games which lie above the $0$-dimensional subsimplex $\{p_2\}$.  Representations $(q: p_2)$ lie in $\langle 2 \rangle$ for $q \in (0, 1/2]$ and lie in $\langle 32 \rangle$ for $q \in (1/2, 1]$.

The only game in any $\C_n$ that lies above point $p_1$ is the dictator game $\langle n \rangle$.

The simplest nonsymmetric games occur for 3 players; they are $\langle 31 \rangle$ and its dual $\langle 3,21 \rangle$.  Each has power composition $(1,2)$, which is the composition associated with the one-dimensional subsimplex $\overline{p_1p_3} \subset \Delta_3$.  We depict the face $\overline{p_1p_3}$ of $\C_3$ and the polytopes which intersect it in Figure~\ref{fig:p1p3}.
 \hfill{$\Diamond$}
\end{example}

\begin{center}
\begin{figure}
\includegraphics[width=0.5\textwidth]{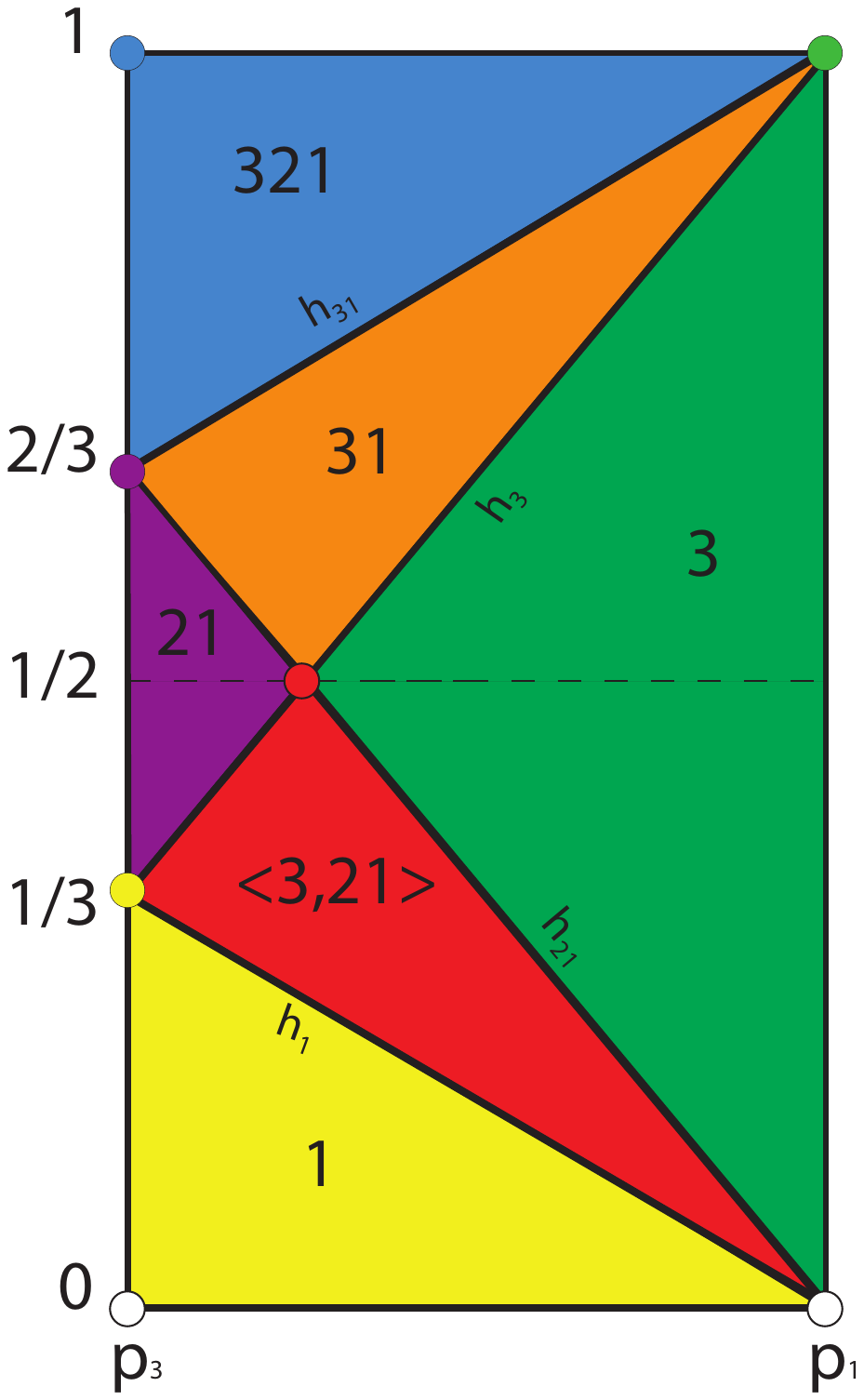}
\caption{From Example \ref{ex:n=3}, the face $\overline{p_1p_3}$ of $\C_3$ is shown along with all polytopes which intersect it (labeled by their corresponding game).  Notice (cf.\  Theorem~\ref{thm:dualgeom}) that the interiors of dual games are reflected about the hyperplane $q=1/2$.}
\label{fig:p1p3}
\end{figure}
\end{center}

\subsection{A geometric view of weighted voting posets}
Our last results demonstrate that the geometric viewpoint of weighted games via their polytopes is highly correlated to both hierarchies and the poset of weighted games.  In this section we prove that the polytopes in $\C_n$ are situated according to poset $\w_n$.  Furthermore, we show that the facets of $P_v$ correspond to covering relations in $\w_n$ and to the hierarchy of players in $v$.

\begin{definition}
A weight vector is called \emph{generic} if for all coalitions $A$ and $B$, we have $$w_A=w_B \iff A=B.$$
\end{definition}

\begin{theorem} \label{thm:satchain}
Given a generic weight vector $\ww$, consider the vertical line in $\C_n$ above $\ww$.  As the quota increases, the games traversed form a saturated chain in $\w_n$, the poset of weighted games.  Moreover, the chains
\begin{enumerate} 
\item are maximal:  each one begins with the game $\langle 1 \rangle$ (of unique minimal rank $1$) and finishes with consensus rule $\langle [n] \rangle$ (of unique maximal rank $2^n-1$); 
\item are \emph{self-dual}:  if game $v$ is in the chain, so too is $v^*$.
\end{enumerate}
\end{theorem}

We note that not every saturated chain corresponds to a vertical line segment above a generic point.  Obstructions beyond self-duality exist, but these are not fully understood.  Every saturated chain does correspond to some piecewise-linear motion through polytopes, as we will see in Corollary~\ref{cor:pl}.  In Section~\ref{sec:future}, we discuss some approaches for studying saturated chains in the posets $\w_n$ and $J_n$.   

\begin{proof}
Given a generic weight vector $\ww$, all representations $(q:\ww)$ lie in game $\langle 1 \rangle$ for $q \in (0,w_1]$.  Similarly, for $q \in (1-w_1, 1]$ the representations lie in the consensus rule game $\langle [n] \rangle$.

Suppose we are at a representation $z$ in game $v$.  Moving upwards from $z$, we remain in $v$ until we encounter the next lowest hyperplane $h_A$ at $q=w_A$.  As we cross this hyperplane, coalition $A$ changes from winning to losing, which means we move into a game $u$ with winning coalitions $W_u = W_v \setminus \{A\}$.  This is precisely what it means to say that $u$ covers $v$ in $\w_n$.

Since the weight vector $\ww$ is generic, we will never encounter two hyperplanes at once (else $q=w_A=w_B$).  Thus, as we increase the quota, representation $(q:\ww)$ crosses, one by one, each of the $2^n - 2$ hyperplanes $h_A$ corresponding to nonempty coalitions other than $[n]$.  After crossing $h_A$, coalition $A$ switches from winning to losing, and we lie in a new game with rank one greater than before.  We continue until finally we arrive in the consensus rule game.  Thus our vertical line corresponds to a maximal saturated chain.  The duality of the chain follows immediately from Theorem~\ref{thm:dualgeom}.
\end{proof}

The hierarchy of $v$ and its position within the poset $\w_n$ determine which facets occur for the polytope $P_v$.

\begin{theorem}{\label{thm:facets}}
Let $v$ be a weighted game whose $n$ players form $k$ nontrivial symmetry classes.  Let $d$ represent the degree of $v$ in the (Hasse diagram of) poset $\w_n$.  Then, for $v \not= \langle [n] \rangle$ and $v \not=\langle 1 \rangle$, the polytope $P_v$ has $n-k+d$ facets.  (For the excluded games $v=\langle [n] \rangle$ and $v =\langle 1 \rangle$, the polytope $P_v$ has $n+1$ facets.)
\begin{enumerate}[\; 1. \;]
\item The top facets of polytope $P_v$ are in one-to-one correspondence with the weighted games $u_i$ that cover $v$ in $\w_n$, except in the case of consensus rule $v=\langle [n] \rangle$, which has one top facet $\{q=1\}$.  The facet is a subset of hyperplane $h_A$, where $A$ is the one coalition winning in $v$ but not $u_i$.  Coalition $A$ is  shift-minimal for $v$, i.e. it is a generator of the filter of winning coalitions for $v$ in $M(n)$.

\medskip
\item The bottom facets of polytope $P_v$ are in one-to-one correspondence with the weighted games $g_i$ that are covered by $v$ in $\w_n$, unless $v = \langle 1 \rangle$, which has one bottom facet $\{q=0\}$.  The facet is a subset of hyperplane $h_B$, where $B$ is the one coalition winning in $g_i$ but not $v$.  Coalition $B$ is a shift-maximal losing coalition for $v$, i.e., it is a generator of the order ideal of losing coalitions for $v$ in $M(n)$.

\medskip
\item There exist $n-k$ vertical facets in polytope $P_v$.  Each lies above a subsimplex of $\Delta_n$ given by $w_{i+1} = w_i$ or $w_1=0$.
\end{enumerate}
\end{theorem}

\begin{proof}
We begin with the first statement.  The case of consensus rule, i.e., when $v = \langle [n] \rangle$, is immediate, since it has one top facet, one bottom facet $[n] \setminus 1$, and intersects all vertical facets, except $\{w_1=0\}$, so $n-1$ vertical facets.  Therefore we may assume $v \neq \langle [n] \rangle$.  Let $F$ be a top facet of $P_v$; we will describe the unique game corresponding to $F$.  There is a hyperplane $h_A$ which contains $F$; it is unique since hyperplanes arising from distinct coalitions have distinct normal vectors.  Consider a point $z$ in the interior of $F$; without loss of generality we may assume the weights in $z$ are generic.  Moving vertically from below $z$ to above $z$ changes coalition $A$ from winning to losing.  If we travel upwards by a small enough amount so as to not cross any other hyperplane, we ensure that no other coalition changes its status.  Thus the points immediately above $z$ lie in the game whose winning coalitions are $W_v \setminus \{A\}$; such a game covers $v$ in $\w_n$.

Now suppose we have a weighted game $u$ which covers $v$ in $\w_n$, i.e., $W_u = W_v \setminus \{A\}$.  The corresponding polytopes $P_u$ and $P_v$ lie on the same side of all other hyperplanes $h_B$ ($B \neq A$).  These polytopes are top-dimensional and distinct by Proposition~\ref{ndim}, so they must be separated by the hyperplane $h_A$.  We are guaranteed that $h_A$ intersects $P_v$ in a facet by the existence of generic points in their intersection.  Thus, the game $u$ corresponds to a unique facet of $P_v$.

Observe that the second statement in the theorem is merely the corresponding restatement of the first one from the point of view of the greater coalition rather than the lesser coalition.  The arguments above prove this statement as well.

Now we consider the third statement.  The region of allowable weights $\Delta_n$ is an $(n-1)$-dimensional simplex, so it has $n$ facets, which are given by equations $w_{i+1} = w_i$ ($1 \leq i \leq n-1$) and $w_1=0$. In proving Lemma~\ref{EqualWeights} we concluded that polytope $P_v$ contains points where any two consecutive players $\v{i}$ and $\v{i+1}$ in the same symmetry class have equal weights.  Further, these points may be chosen so that the weights are otherwise generic.  Indeed we get an $(n-1)$-dimensional set of such points, all of which lie in the interior of the corresponding vertical facet of $P_v$.

Thus, if players $\v{i}$ and $\v{i+1}$ lie in the same symmetry class, then $P_v$ contains a vertical facet over $w_i=w_{i+1}$; the converse is clearly also true.  Similarly, player $\v{1}$ is a dummy if and only if $P_v$ contains points over $w_1=0$.  For $k$ different nontrivial symmetry classes, there are $n-k$ of these players and hence $n-k$ vertical facets.

The degree $d$ of $v$ in $\w_n$ is equal to the number of covers of $v$ plus the number of weighted games covered by $v$.  Therefore $d$ is equal to the number of top facets plus the number of bottom facets.  These facets together with the $n-k$ vertical facets comprise the $n-k+d$ facets of the polytope $P_v$.
\end{proof}

This theorem tells us a great deal about how the structure of $\w_n$ arises in $\C_n$.

\begin{corollary} \label{cor:pl}
Every saturated chain of games in $\w_n$ may be achieved by some piecewise linear motion through $\C_n$.
\end{corollary}

\begin{remark}
Since every $n$-dimensional polytope has at least $n+1$ facets, we may immediately conclude that the degree $d$ of $v\in\w_n$ is greater than the number of symmetry classes $k$.  When $d=k+1$, polytope $P_v$ is a simplex (i.e. $P_v$ has $n+1$ facets).  We note that all games of $4$ or fewer players have a simplex as their polytope; so too do 101 out of 117 games with $5$ players.  The exceptions are the following 8 proper games and their duals: $\langle 541,5321 \rangle, \langle 541, 4321 \rangle,$ 
$\langle 541, 532, 4321 \rangle,$  
$\langle 531, 4321 \rangle, \langle 54, 531, 4321 \rangle,$ 
$\langle 521, 4321 \rangle,\langle 54, 521, 4321 \rangle, \langle 53, 521, 4321 \rangle$.
\hfill{$\Diamond$}
\end{remark}

An $n$-dimensional polytope is \emph{simple} if it has $n$ facets meeting at every vertex, e.g., a cube is simple; an octahedron is not.  Every simplex is simple.  However, not all polytopes $P_v$ are simple.  As an example, the polytope for game $v=\langle 521, 4321 \rangle$ on $n=5$ players has seven facets (two top, two bottom, three vertical); six of them meet at the vertex $\left(\tfrac{3}{5}: \tfrac{2}{5}, \tfrac{1}{5}, \tfrac{1}{5}, \tfrac{1}{5}, 0 \right)$.

\subsection{When unweighted games cover weighted games} \label{sec:unweightedmethod}
From Theorem~\ref{thm:facets}, we obtain a geometric understanding of the distinction between weighted and unweighted games.  After considering another corollary of this theorem, we detail a method for determining the weightedness of linear games which either cover or are covered by a weighted game in $J_n$.  

\begin{corollary} \label{cor:unweighted}
If a weighted game $v$ has $k$ more shift-minimal winning coalitions than top facets in its polytope, then $k$ of the games covering $v$ in $J_n$ are unweighted.  Similarly, if $v$ covers  $\ell$ more games in $J_n$ than it has bottom facets in its polytope, then $\ell$ of the covered games are unweighted.  
\end{corollary}

\begin{proof}
Theorem~\ref{thm:facets} states that every weighted game which covers a weighted game $v$ corresponds to a unique top facet of $P_v$.  Each cover of $v$ in $J_n$ is obtained by removing a shift-minimal winning coalition; see Proposition~\ref{prop:covers}.   Therefore, shift-minimal winning coalitions can be partitioned into two classes; those which correspond to a top facet and those which do not.  The $k$ shift-minimal winning coalitions which do not correspond to a top facet correspond to the $k$ unweighted games which cover $v$.

Similarly, the shift-maximal losing coalitions are in one-to-one correspondence with games in $J_n$ covered by $v$.  Thus, they can be partitioned into two classes; those which correspond to a bottom facet and those which do not.  Therefore the $\ell$ shift-maximal losing coalitions which do not correspond to a bottom facet correspond to the $\ell$ unweighted games covered by $v$.
\end{proof}

Geometrically, an unweighted game $u$ covers a weighted game $v$ in $J_n$ precisely when one generator fails to be an active constraint in defining the polytope.  

\begin{theorem} \label{thm:method}
Let $u$ be a linear game covering a weighted game $v$ in $J_n$ and assume $W_u = W_v \setminus \{ A \}$.  Then $u$ is weighted if and only if there exists a weight vector $\ww$ in the footprint $\pi(P_v)$ such that $w_A < w_B$ at $\ww$ for all other winning coalitions $B$ of $v$.  

Similarly, if a weighted game $u$ covers a linear game $v$ in $J_n$ in such a way that $W_u = W_v \setminus A$, then $v$ is weighted if and only if there exists a weight $\ww$ in the footprint $\pi(P_u)$  such that $w_A > w_C$ at $\ww$ for all other losing coalitions $C$ of $u$.  
\end{theorem}

We note that it suffices to check only the shift-minimal winning coalitions of $v$ in the first statement and only the shift-maximal losing coalitions of $u$ in the second.

\begin{proof}
Assume $u$ is weighted.  This means $u$ covers $v$ in $\w_n$, so by Theorem~\ref{thm:facets}, hyperplane $h_A$ forms a top facet of the polytope $P_v$.  We claim that for any point $z=(q:\ww)$ on the interior of this top facet, the weight $w_A$ must be less than the weight of all other winning coalitions in $v$.  Note that point $z$ is a realization of $v$.  If $w_B < w_A$ at $z$, where $q=w_A$, then the weight of $B$ is less than the quota, so $B$ is losing in $v$.  If $w_B = w_A$ at $z$, then the point $z$ is not actually on the interior of the facet; rather, $z$ lies on the face where $P_v$ intersects $h_A \cap h_B$.  Thus, the claim holds and we have proven one direction of the first statement.

To prove the other direction, assume there exists a weight $\ww \in \pi(P_v)$ where $w_A < w_B$ for all other winning coalitions $B$ of $v$.  Since $\ww$ lies in the footprint of $P_v$, there exists some realization $(q:\ww)$ of $v$.  Choosing a quota $q'$ greater than $w_A$ and less than the minimum of all weights $w_B$ for $B \in W_v \setminus \{A\}$ guarantees that we have realized the game whose winning coalitions are precisely $W_v \setminus \{A\}$, namely game $u$.  Thus we have shown $u$ is weighted, which finishes the proof of the first statement.

The proof of the second statement in the theorem is directly analogous.
\end{proof}

\begin{example} \label{ex:unweighted}
Consider the weighted game $v=\langle 987, 8741 \rangle$ mentioned in Section~\ref{sec:properties}.  It has two generators, but if $987$ becomes a losing coalition, then the unweighted linear game $u=\langle 8741 \rangle$ results.  By Theorem~\ref{thm:facets}, the polytope $P_v$ has only one top facet.
Though we know (by trade robustness) that $\langle 8741 \rangle$ is unweighted, we reprove it here using Theorem~\ref{thm:method}.  That is, we prove that there does not exist a weight $\ww$ satisfying $w_{987}<w_B$ for all other winning coalitions $B$ of $v$.

We argue that if the players' weights are restricted to lie in the footprint of $v$, then the generators of $v$ are actually comparable there:  we are going to prove that coalition $987$ is stronger than $8741$ among all weights in $\pi(P_v)$.  The weights of these two coalitions are equal along the set $S=\{\ww \, \left| \, w_9 = w_4 + w_1 \right. \}$.  Our strategy is to show $S$ is disjoint from $\pi(P_v)$; this implies that $w_{987}$ can never be equal to $w_{8741}$ in $\pi(P_v)$ and hence if one of these two coalitions is ever greater than the other in $\pi(P_v)$ it will remain greater in $\pi(P_v)$.    For any choice of weights $\ww$ in $S$, coalitions $9752$ and $75421$ must have the same weight.  The former is winning in $v$, while the latter is losing, so no weights from $S$ can form a represention of $v$.  Thus, $987$ and $8741$ are  comparable above the footprint $\pi(P_v)$ of $v$; the representation $(22:9,9,9,3,3,3,1,1,1)$ of $v$ shows that $987$ is stronger than $8741$ there.

Had we no \emph{a priori} knowledge of the weightedness of $\langle 8741 \rangle$, observe that the preceding paragraph would be sufficient to determine that it is unweighted.
\hfill{$\Diamond$}
\end{example}

We might ask, is Theorem~\ref{thm:method} a useful method for determining weightedness? Quite possibly, for the relevant games.  Understanding relevancy raises the question of determining how many games in $J_n$ either cover or are covered by a weighted game.   Though this footprint method still results in a linear programming (LP) problem, it is one that is different and possibly easier than the traditional LP problem of determining the existence of weights so that all generators are greater than all shift-maximal losing coalitions.  One slight drawback is the reliance upon knowledge of the poset $J_n$; while outputting this poset is computationally infeasible for more than a small number of players, obtaining local knowledge of how game $v$ sits inside $J_n$ is more straightforward.  A challenge for future work is to efficiently implement this as an algorithm for computation.


\section{Future directions}
\label{sec:future}

Our hope is that the combinatorial (poset) and geometric (polytope) approaches to linear games that we describe herein will lead to a greater understanding of linear games and weighted voting.  Many natural questions remain to be answered about these structures, some of which we have already mentioned (e.g., saturated chains, implementing the method of Theorem~\ref{thm:method}).

One direction for further study is the connection between the geometry of $\C_n$ and power distributions for weighted games.  In our upcoming paper \cite{gpi}, we define a geometrically-based, monotonic power index on all weighted games which has several useful properties.

Another avenue for further investigation is the classification of the maximal saturated chains.  Every saturated chain $\gamma$ in poset $J_n$ produces an ordering on all shift-minimal winning coalitions for the games in $\gamma$.  This is a linear ordering, except for the generators of the highest game $g$ in $\gamma$; these are strictly greater than all other generators but incomparable to each other.

For example, consider the following chain in the linear games poset $J_5^+$: 
$$\langle 54, 531 \rangle < \langle 54, 532 \rangle < \langle 541, 532 \rangle < \langle 532 \rangle < \langle 542, 5321 \rangle < \langle 543, 5321 \rangle.$$  
This chain places the following linear ordering on the generators that it contains:  
$$531 < 54 < 541 < 532 < 542 < 5321.$$  
Note that the first inequality implies $31<4$ whereas the last one implies $4<31$; no set of weights could possibly accomplish this.  Thus, this particular chain does not correspond to a vertical line within the weighted voting polygon.  We would like to understand which saturated chains in $\w_n$ correspond to vertical line segments in $\C_n$.

\begin{definition}
A saturated chain in $\w_n$ is said to be a \emph{weighted chain} if there exists a weight vector $\ww$ such that each game $g$ in the chain can be realized using these weights and some quota $q_g$. 
\end{definition}

Recall, a chain is said to be \emph{self-dual} if for every element $v$ in the chain, its dual is also in the chain (cf.\ Theorem~\ref{thm:satchain}).
Clearly any weighted chain must be self-dual, since duality in the polytope corresponds to reflection across the hyperplane given by $q=0.5$.

We conjecture a necessary and sufficient condition for a saturated chain to be a weighted chain.  This condition involves a property called \emph{inequality robustness} (which we will not define here) that is closely related to the notion of trade robustness.

\begin{conjecture} \label{conj:chain}
A saturated chain in $\w_n$ is a weighted chain if and only if the chain is inequality robust.
\end{conjecture}

Among proper games, saturated chains which are not weighted chains first arise with $5$ players.  For $\Pi_4$ (which equals $J_4^+$ and $\w_4^+$), there are 14 distinct maximal saturated chains, all of which are inequality robust and correspond to vertical line segments in $\C_4$ and thus are weighted chains.

\section*{Acknowledgements}
This project has its origins in a seminar taught by the second author.  One of the seminar students, Annalaissa Johnson, contributed greatly to our overall understanding of weighted voting  and power.  The authors are grateful to Curtis Greene for his Mathematica package ``Posets" which assisted with our computations and visualizations.  The authors would also like to thank an anonymous referee for suggesting the method used to prove Proposition~\ref{prop:induced}.

\bibliographystyle{plain}
\bibliography{GCVotingBib}

\pagebreak

\appendix
\section{Unweighted linear games with $6$ players} \label{app:unweighted6}

Table~\ref{n=6} lists all unweighted linear games with $6$ players in $J_6^+$.  The left-hand side lists 20 games with minimal rank $32$ in $J_6^+$, while the right-hand-side lists 20 games with rank greater than $32$.  The remainder of the unweighted linear games with $6$ players are obtained by taking the duals of the unweighted linear games on the right-hand-side.  We note that all 60 unweighted games in $J_6$ are improper.

\renewcommand{\arraystretch}{1.4}
\begin{table}[ht]
\begin{center}
\begin{tabular}{ l  c   l  c }
\toprule
Rank 32 games & \qquad & Higher rank games & Rank  \\
\cmidrule{1-1} \cmidrule{3-4} 
621, 542 &   & 621, 543, 5421 & 33 \\
    
621, 543, 5321 &  & 631, 542 & 33 \\
    
63, 5421 &   & 632, 541 & 33 \\
    
631, 541 &   & 64, 4321 & 33 \\
    
631, 542, 5321 &   & 64, 543, 5321 & 33 \\
    
632, 541, 5321 &   & 65, 542, 4321 & 33 \\
    
64, 542, 5321 &   & 65, 621, 543 & 33 \\
   
64, 543, 4321 &   & 65, 632, 5321 & 33 \\
   
64, 621, 543 &   & 65, 632, 543, 4321 & 33 \\
   
64, 631, 5321 &   & 65, 641, 543, 4321 & 33 \\
   
64, 632, 4321 &   & 65, 641, 632, 4321 & 33 \\

64, 632, 543, 5321 &   & 621, 543 & 34 \\  
 
641, 532 &   & 64, 5321 & 34 \\ 
  
65, 621, 543, 5421 &   & 65, 632, 4321 & 34 \\
   
65, 631, 4321 &   & 65, 641, 4321 & 34 \\
  
65, 631, 542 &   & 65, 642, 543, 4321 & 34 \\

65, 632, 541 &   & 65, 543, 4321 & 35 \\
65, 632, 542, 4321 &   & 65, 642, 4321 & 35 \\
  
65, 641, 542, 4321 &   & 65, 643, 4321  & 36 \\

65, 641, 632, 543, 4321 & \hspace*{1in}  & 65, 4321 & 37 \\
\bottomrule
\end{tabular}
\end{center}
\smallskip
\caption{This table lists, in terms of their generators, 40 of the 60 unweighted linear games with 6 players.  The remaining 20 are the dual games of the ones in the second column.}
{\label{n=6}}
\end{table}
\renewcommand{\arraystretch}{1}

\end{document}